\documentclass[12pt,oneside,english]{amsart}
\usepackage[T1]{fontenc}
\usepackage[latin9]{inputenc}
\usepackage[a4paper]{geometry}
\usepackage{babel}
\usepackage{prettyref}
\usepackage{mathrsfs}
\usepackage{mathtools}
\usepackage{amstext}
\usepackage{amsthm}
\usepackage{amssymb}
\usepackage{esint}
\usepackage[unicode=true,
 bookmarks=false,
 breaklinks=false,pdfborder={0 0 1},backref=section,colorlinks=false]
 {hyperref}

\makeatletter
\numberwithin{equation}{section}
\numberwithin{figure}{section}
\theoremstyle{plain}
\newtheorem{thm}{\protect\theoremname}
\theoremstyle{plain}
\newtheorem{cor}[thm]{\protect\corollaryname}
\theoremstyle{plain}
\newtheorem{lem}[thm]{\protect\lemmaname}
\theoremstyle{plain}
\newtheorem{prop}[thm]{\protect\propositionname}
\theoremstyle{remark}
\newtheorem{rem}[thm]{\protect\remarkname}

\usepackage{babel}

\makeatother

\providecommand{\corollaryname}{Corollary}
\providecommand{\lemmaname}{Lemma}
\providecommand{\propositionname}{Proposition}
\providecommand{\remarkname}{Remark}
\providecommand{\theoremname}{Theorem}

\begin{document}
\title[Eta remainder]{Hyperbolicity, irrationality exponents and the eta invariant  }
\author{Nikhil Savale}
\begin{abstract}
We consider the remainder term in the semiclassical limit formula
of \cite{Savale-Gutzwiller} for the eta invariant on a metric contact
manifold, proving in general that it is controlled by volumes of recurrence
sets of the Reeb flow. This particularly gives a logarithmic improvement
of the remainder for Anosov Reeb flows, while for certain elliptic
flows the improvement is in terms of irrationality measures of corresponding
Floquet exponents.
\end{abstract}

\address{Universität zu Köln, Mathematisches Institut, Weyertal 86-90, 50931
Köln, Germany}
\email{nsavale@math.uni-koeln.de}
\thanks{N. S. is partially supported by the DFG funded project CRC/TRR 191.}
\maketitle

\section{Introduction}

The eta invariant of Atiyah-Patodi-Singer \cite{APSI} is formally
the signature of a Dirac operator and appears as a correction term
in the index theorem for manifolds with boundary. Much like the signature
of a matrix, it is not in general continuous in the operator making
it difficult to understand its behavior/asymptotics in parameters/limits.
In previous work \cite{Savale-thesis2012,Savale-Asmptotics,Savale2017-Koszul,Savale-Gutzwiller}
the author considered the asymptotics of the eta invariant of a coupled
Dirac operator in the semiclassical limit. In particular \cite[Thm 1.2]{Savale-Gutzwiller}
proved a semiclassical limit formula for the eta invariant on metric-contact
manifold whose Reeb flow satisfies a non-resonance condition, i.e.
rational independence of relevant Floquet exponents. In this article
we consider the question of the remainder term in this formula, relating
it to further dynamical properties such as volumes of recurrence sets.
This particularly gives a logarithmic improvement of the remainder
when the flow is Anosov. While for certain elliptic flows the remainder
is related to finer properties such as irrationality measures and
Diophantine approximation of the Floquet exponents.

Let us state our results precisely. Let $\left(X,g^{TX}\right)$ be
a compact, spin Riemannian manifold of odd dimension $n=2m+1$. Let
$a\in\Omega^{1}\left(X\right)$ be a contact one form satisfying $a\wedge\left(da\right)^{m}\neq0$.
This gives rise to the contact hyperplane $H=\textrm{ker}\left(a\right)\subset TX$
as well as the Reeb vector field $R$ defined by $i_{R}da=0$, $i_{R}a=1$.
The metric $g^{TX}$ is assumed to be $strongly\:suitable$ to the
contact form $a$ in the following sense: the contracted endomorphism
$\mathfrak{J}:T_{x}X\rightarrow T_{x}X$, defined via $da\left(.,.\right)=g^{TX}\left(.,\mathfrak{J}.\right)$,
has spectrum $\textrm{Spec}\left(\mathfrak{J}_{x}\right)=\left\{ 0\right\} \cup\left\{ \pm i\mu_{j}\right\} _{j=1}^{m}$
independent of the point $x\in X$. This hypothesis includes all metric
contact manifolds, in which case $\mathfrak{J}$ is a compatible almost
complex structure on $H$ and satisfies $\mathfrak{J}^{2}=-1$.

Next let $\left(L,h^{L}\right)\rightarrow X$ be another complex Hermitian
line bundle on $X$ and $A_{0}$ a unitary connection on $L$. This
gives rise to the family of unitary connections $A_{h}\coloneqq A_{0}+\frac{i}{h}a$,
$h\in\left(0,1\right]$, and corresponding coupled Dirac operators
\begin{equation}
D_{h}\coloneqq hD_{A_{h}}=hD_{A_{0}}+ic\left(a\right):\:C^{\infty}\left(X;S\otimes L\right)\rightarrow C^{\infty}\left(X;S\otimes L\right),\label{eq:Semiclassical Magnetic Dirac}
\end{equation}
with $S$, $c$ denoting the corresponding spin bundle and Clifford
multiplication endomorphism respectively. The above being an elliptic
operator has a discrete spectrum whose behavior as $h\rightarrow0$
is of our interest. In particular, we shall be interested in the asymptotics
of the eta invariant $\eta_{h}=\eta\left(D_{h}\right)$, formally
the signature, of the Dirac operator (see \prettyref{subsec:Spectral-invariants-of}
below). To state our result, define the endomorphisms $\left(\nabla^{TX}\mathfrak{J}\right)^{0}:R^{\perp}\rightarrow R^{\perp}$,
$\left|\mathfrak{J}\right|:R^{\perp}\rightarrow R^{\perp}$, via
\begin{align}
\left(\nabla^{TX}\mathfrak{J}\right)^{0}v\coloneqq & \left(\nabla_{v}^{TX}\mathfrak{J}\right)R,\quad\forall v\in R^{\perp},\nonumber \\
\left|\mathfrak{J}\right|\coloneqq & \sqrt{-\mathfrak{J}^{2}}.\label{eq: abs. ctct. end.}
\end{align}
Let $T_{0}$ be the shortest period of the Reeb vector field. For
each $\varepsilon,T>0$ we define the recurrence set
\begin{equation}
S_{T,\varepsilon}\coloneqq\left\{ x\in X|\exists t\in\left[\frac{1}{2}T_{0},T\right]\textrm{ s.t. }d^{g^{TX}}\left(e^{tR}x,x\right)\leq\varepsilon\right\} \label{eq:recurrence set}
\end{equation}
as well as its extended neighborhood $S_{T,\varepsilon}^{e}\coloneqq\left\{ x\in X|d^{g^{TX}}\left(x,S_{T,\varepsilon}\right)\leq\varepsilon\right\} $.
Here $d^{g^{TX}}$ denotes the Riemannian distance on $X$ and we
also denote by $\mu_{g^{TX}}$ be the Riemannian volume below. \\
Next we further, for each $\ell>0$, define the Ehrenfest time via
\begin{align}
T_{E}^{\ell}\left(h\right) & \coloneqq\begin{cases}
\infty; & \textrm{if }\;\limsup_{t\rightarrow\infty}J_{t}<\infty\\
\frac{\left|\ln h\right|}{\Lambda_{\textrm{max}}+\ell}; & \textrm{otherwise},
\end{cases}\label{eq:Ehrenfest time}\\
\textrm{ where}\quad J_{t} & \coloneqq\sup_{x\in X}\ln\left|de^{tR}\left(x\right)\right|\label{eq:max norm Jacobian}\\
\textrm{ and }\quad\Lambda_{\textrm{max}} & \coloneqq\limsup_{t\rightarrow\infty}t^{-1}J_{t}\label{eq:max expansion rate}
\end{align}
are the sup norm of the Jacobian and the maximum expansion rate of
the Reeb flow respectively.  

Our first result is now the following. 
\begin{thm}
\label{thm: eta semiclassical limit}Let $a$ be a contact form and
$g^{TX}$ a strongly suitable metric. For each $\varepsilon=ch^{\delta}$,
$\delta\in\left[0,\frac{1}{2}\right)$, $c,\ell>0$ and $T\leq\left(\frac{1}{2}-\delta\right)T_{E}^{\ell}\left(h\right)$,
 the rescaled eta invariant of the Dirac operator satisfies \footnote{The implicit constant in the big $O$ notation in \prettyref{eq:general remainder}
is independent of the semiclassical parameter $h$, and possibly dependent
on all the other quantities $\left(X,g^{TX},a,L,h^{L},\varepsilon,\ell,\delta\right)$.} 
\begin{align}
h^{m}\eta\left(D_{h}\right) & =-\frac{1}{2}\frac{1}{\left(2\pi\right)^{m+1}}\frac{1}{m!}\int_{X}\left[\textrm{tr }\left|\mathfrak{J}\right|^{-1}\left(\nabla^{TX}\mathfrak{J}\right)^{0}\right]a\wedge\left(da\right)^{m}+R\left(h\right)\label{eq: eta formula}\\
\textrm{ where }\qquad\left|R\left(h\right)\right| & \leq\frac{\det\left|\mathfrak{J}\right|}{\left(4\pi\right)^{n/2}}T^{-1}+O\left(T^{-2}+\mu_{g^{TX}}\left(S_{T,\varepsilon}^{e}\right)\right).\label{eq:general remainder}
\end{align}
\end{thm}

As noted in \cite{Savale-Gutzwiller}, the leading term in \prettyref{eq: eta formula}
equals a multiple of the volume $-\frac{m}{2}\frac{1}{\left(2\pi\right)^{m+1}}\textrm{vol}\left(X\right)$
in the case of a metric-contact manifold. 

The formula \prettyref{eq:general remainder} above can be made more
explicit in specializations based on the estimates for the recurrence
set \prettyref{eq:recurrence set}. The first such corollary of the
above is obtained by letting $\varepsilon,T$ be $h-$independent.
As $\varepsilon\rightarrow0$ the volume of the recurrence set approaches
the measure of the set of closed trajectories having period at most
$T$. Letting $T\rightarrow\infty$ then gives the following.
\begin{cor}
\label{cor: Cor-DG}Let $a$ be a contact form and $g^{TX}$ a strongly
suitable metric. Assuming the set of closed Reeb trajectories is of
measure zero, the rescaled eta invariant of the Dirac operator satisfies
\begin{align}
h^{m}\eta\left(D_{h}\right) & =-\frac{1}{2}\frac{1}{\left(2\pi\right)^{m+1}}\frac{1}{m!}\int_{X}\left[\textrm{tr }\left|\mathfrak{J}\right|^{-1}\left(\nabla^{TX}\mathfrak{J}\right)^{0}\right]a\wedge\left(da\right)^{m}+o\left(1\right).\label{eq: eta formula-1}
\end{align}
\end{cor}

The next specialization is when the Reeb flow is Anosov, whereby the
union of closed Reeb trajectories automatically has null measure.
The recurrence set in this case can be shown to satisfy an exponential
estimate in time in terms of the topological entropy $\mathtt{h}_{\textrm{top}}$
of the flow (see \prettyref{subsec:Anosov-flows} below). Our main
\prettyref{thm: eta semiclassical limit} then has the following corollary.
\begin{cor}
\label{cor:Cor-Berard}Let $a$ be a contact form and $g^{TX}$ a
strongly suitable metric. Assuming the Reeb flow of $a$ to be Anosov,
the rescaled eta invariant of the Dirac operator satisfies 
\begin{align}
h^{m}\eta\left(D_{h}\right) & =-\frac{1}{2}\frac{1}{\left(2\pi\right)^{m+1}}\frac{1}{m!}\int_{X}\left[\textrm{tr }\left|\mathfrak{J}\right|^{-1}\left(\nabla^{TX}\mathfrak{J}\right)^{0}\right]a\wedge\left(da\right)^{m}+O\left(\left|\ln h\right|^{-1}\right).\label{eq: eta formula-1-1}
\end{align}
More precisely, the remainder above satisfies the estimate 
\begin{equation}
\left|R\left(h\right)\right|\leq\mathtt{h}_{\textrm{top}}\left(\frac{16}{n}\frac{\det\left|\mathfrak{J}\right|}{\left(4\pi\right)^{n/2}}\right)\left|\ln h\right|^{-1}+o\left(\left|\ln h\right|^{-1}\right)\label{eq:remainder estimate with entropy}
\end{equation}
 in terms of the topological entropy of the flow.
\end{cor}

In cases where the flow is elliptic, the recurrence set can be harder
to control. As a particular case we consider certain elliptic flows
on Lens spaces. In dimension three it is known \cite[Thm 1.2]{Hutchings-Taubes2009}
that any contact manifold, all of whose Reeb orbits are non-degenerate
and elliptic, is necessarily a Lens space. To define these, fix non-negative
integers $q_{0},q_{1},\ldots,q_{m}$, with $q_{0}>1$, as well as
positive reals $a_{0},\ldots,a_{m}$ such that $\left(a_{0}^{-1}a_{1},\ldots,a_{0}^{-1}a_{m}\right)\notin\mathbb{Q}^{m}$.
To such a tuple of reals is associated an irrationality exponent/measure
$\nu\left(a_{1},\ldots,a_{m+1}\right)>1$ of Diophantine approximation
(see \prettyref{subsec:Diophantine-approximation}). The Lens space
is now defined as the quotient
\begin{align}
X & =L\left(q_{0},q_{1},\ldots,q_{m};a_{0},\ldots a_{m}\right)\coloneqq E\left(a_{0},\ldots,a_{m}\right)/\mathbb{Z}_{q_{0}},\quad\textrm{ where }\label{eq: irrational lens space}\\
E\left(a_{0},\ldots,a_{m}\right) & \coloneqq\left\{ \left(z_{0},\ldots,z_{m}\right)\in\mathbb{C}^{m+1}|\sum_{j=0}^{m}a_{j}\left|z_{j}\right|^{2}=1\right\} \label{eq:irrational ellipsoids}
\end{align}
is the irrational ellipsoid. Above the $\mathbb{Z}_{q_{0}}$ action
on the ellipsoid is given by $e^{\frac{2\pi i}{q_{0}}}\left(z_{0},\ldots,z_{m}\right)=\left(e^{\frac{2\pi i}{q_{0}}}z_{0},e^{\frac{2\pi iq_{1}}{q_{0}}}z_{1},\ldots,e^{\frac{2\pi iq_{m}}{q_{0}}}z_{m}\right)$.
The contact form is chosen to be 
\begin{equation}
a=\left.\sum_{j=0}^{m}\left(x_{j}dy_{j}-y_{j}dx_{j}\right)\right|_{E\left(a_{0},\ldots,a_{m}\right)},\quad z_{j}=x_{j}+iy_{j},\label{eq:tautological one form}
\end{equation}
the restriction of the tautological form on $\mathbb{C}^{m+1}$, which
can be seen to be $\mathbb{Z}_{q_{0}}$-invariant and hence descending
to the Lens space quotient. We now have the following corollary of
\prettyref{thm: eta semiclassical limit}, below we denote by $O\left(h^{\alpha\pm}\right)$
a term which is $O\left(h^{\alpha\pm\varepsilon}\right)$, $\forall\varepsilon>0$.
\begin{cor}
\label{cor:Cor -Irr}Let $a$ be the tautological contact form and
$g^{TX}$ a strongly suitable metric on the Lens space $X=L\left(q_{0},q_{1},\ldots,q_{m};a_{0},\ldots,a_{m}\right)$
\prettyref{eq: irrational lens space}. The rescaled eta invariant
of the Dirac operator satisfies 
\begin{align}
h^{m}\eta\left(D_{h}\right) & =-\frac{1}{2}\frac{1}{\left(2\pi\right)^{m+1}}\frac{1}{m!}\int_{X}\left[\textrm{tr }\left|\mathfrak{J}\right|^{-1}\left(\nabla^{TX}\mathfrak{J}\right)^{0}\right]a\wedge\left(da\right)^{m}+O\left(h^{\frac{1}{2\nu-1}-}\right)\label{eq: eta formula-1-1-1}
\end{align}
where $\nu=\nu\left(a_{0},a_{1}-\frac{q_{1}}{q_{0}}a_{0},\ldots,a_{m}-\frac{q_{m}}{q_{0}}a_{0}\right)$
denotes the irrationality exponent of the given tuple.
\end{cor}

Our final result is quantum ergodicity for the Dirac operator \ref{eq:Semiclassical Magnetic Dirac},
conditional on ergodicity of the Reeb flow. We refer to \prettyref{thm:QE theorem}
in \prettyref{sec:Quantum-ergodicity} for the precise statement.
Although it is not the main interest here, it is included to show
the direct connection of the author's work with the results of \cite{Colin-de-Verdiere-Hillairet-TrelatI}.
The final \prettyref{thm:QE theorem} is the semiclassical Dirac operator
analogue of the main result therein.

The Corollary \prettyref{cor: Cor-DG} already improves \cite[Thm 1.2]{Savale-Gutzwiller}
proving the same limit formula \prettyref{eq: eta formula-1} therein
under the weaker assumption of null measure of the closed trajectories,
thus akin to the remainder improvement in the Weyl law of Duistermaat-Guillemin
\cite{Duistermaat-Guillemin}. However it should be noted that this
weaker dynamical assumption in Corollary \prettyref{cor: Cor-DG}
is insufficient to obtain the Gutzwiller trace formula of \cite[Thm 1.1]{Savale-Gutzwiller}.
The next Corollary \prettyref{cor:Cor-Berard} is analogous to the
remainder improvement in the Weyl law for the Laplacian by Bérard
\cite{Berard77}. For the Weyl law analogues of the remainder improvements
in \prettyref{thm: eta semiclassical limit} and Corollary \prettyref{cor:Cor -Irr}
we refer to \cite[Sec. 4.5]{Ivrii-newbook-I-2019} and \cite[Ch. 11]{Dimassi-Sjostrand},
although not explicitly written it is inherent in the methods therein.

However, we importantly note that the Fourier integral calculus that
is used in deriving sharp Weyl laws and remainder estimates is largely
unavailable in our context, the basic reason being that the Dirac
operator \prettyref{eq:Semiclassical Magnetic Dirac} is non-scalar
with its principal symbol being non-diagonalizable near its characteristic
variety. The difficulties are in spirit close to those in the analysis
of the wave equation of a hypoelliptic operator with double characteristics
\cite{Melrose-hypoelliptic,Savale-QC}. The method here is based on
the author's earlier work \cite{Savale-thesis2012,Savale-Asmptotics,Savale2017-Koszul,Savale-Gutzwiller},
which in turn uses a combination of Birkhoff normal forms, almost
analytic continuations, microhyperbolicity and local index theory
arguments. New components here include a microlocal trace expansion,
analysis of recurrence sets and a refined Tauberian argument. In \prettyref{subsec:Anosov-flows}
an exponential estimate for the recurrence set of an Anosov flow is
proved based on a seemingly new bound on its topological entropy.
A conjectured characterization of topological entropy is stated, based
on which the estimate \prettyref{eq:remainder estimate with entropy}
in Corollary \prettyref{cor:Cor-Berard} can be improved by a factor
(see Remark \prettyref{rem:Entropy conjecture} below). An analogous
microlocal and second-microlocal trace expansion to the one here was
also derived by the author recently \cite{Savale-QC} in the context
of the hypoelliptic Laplacian of sub-Riemannian geometry. A related
hypoelliptic, second-microlocal Weyl law has also recently appeared
in \cite{TaylorM2020} after \cite{Colin-de-Verdiere-Hillairet-TrelatI}.

The eta invariant asymptotics considered here further has applications
to contact geometry. It was originally motivated by the proof of the
three dimensional Weinstein conjecture using Seiberg-Witten theory
by Taubes \cite{Taubes-Weinstein}. The semiclassical limit formula
for the eta invariant of \cite{Savale-Gutzwiller} was recently used
by the author in \cite{Cristofaro-Gardiner-Savale18} to improve the
remainder term in the asymptotics of embedded contact homology (ECH)
capacities \cite{Cristofaro-Gardiner-Hutchings-Gripp2015}. Our main
theorem here as well as its corollaries could have further applications
in this direction.

The paper is organized as follows. In the first \prettyref{sec:Preliminaries}
we begin with background notions used in the paper including the requisites
on Dirac operators \prettyref{subsec:Spectral-invariants-of}, semiclassical
analysis \prettyref{subsec:The-Semi-classical-calculus} and Diophantine
approximation \prettyref{subsec:Diophantine-approximation}. In \prettyref{sec:Microlocal-trace-expansion}
we prove a microlocal trace expansion for the Dirac operator. This
is later used in \prettyref{sec:Eta-remainder-asymptotics} to prove
our main \prettyref{thm: eta semiclassical limit}. In \prettyref{sec:Examples-of-recurrence}
we consider recurrence sets in the particular Anosov \prettyref{subsec:Anosov-flows}
and elliptic \prettyref{subsec:Irrational-elliptic-flows} cases and
prove the Corollaries \prettyref{cor: Cor-DG}, \prettyref{cor:Cor-Berard}
and \prettyref{cor:Cor -Irr}. In the final \prettyref{sec:Quantum-ergodicity}
we prove quantum ergodicity for the Dirac operator.

\section{\label{sec:Preliminaries}Preliminaries}

\subsection{\label{subsec:Spectral-invariants-of} Spectral invariants of the
Dirac operator}

\noindent We begin by stating the requisites about Dirac operators
used in the paper, \cite{Berline-Getzler-Vergne} provides a standard
reference. Let $\left(X,g^{TX}\right)$ be a compact, oriented, Riemannian
manifold of odd dimension $n=2m+1$. It shall be further equipped
with a spin structure, i.e. a $\textrm{Spin}\left(n\right)$ principal
bundle $\textrm{Spin}\left(TX\right)\rightarrow SO\left(TX\right)$
that is an equivariant double covering of the $SO\left(n\right)$
principal bundle $SO\left(TX\right)$ of orthonormal frames in $TX$.
The unique irreducible representation of $\textrm{Spin}\left(n\right)$
gives rise to the associated spin bundle $S=\textrm{Spin}\left(TX\right)\times_{\textrm{Spin}\left(n\right)}S_{2m}$
. The Levi-Civita connection $\nabla^{TX}$ on the tangent bundle
$TX$ lifts to a connection on $\textrm{Spin}\left(TX\right)$ thus
in turn giving rise to the spin connection $\nabla^{S}$ on the spin
bundle $S$. The Clifford multiplication endomorphism $c:T^{*}X\rightarrow S\otimes S^{*}$
arises from the standard representation of the Clifford algebra of
$T^{*}X$ and satisfies
\begin{align*}
c(a)^{2}=-|a|^{2}, & \quad\forall a\in T^{*}X.
\end{align*}
Next choose $\left(L,h^{L}\right)$ a Hermitian line bundle on $X$
along with $A_{0}$ a unitary connection on it. Given any one-form
$a\in\Omega^{1}(X;\mathbb{R})$ on $X$ we may form the family $\nabla^{h}=A_{h}=A_{0}+\frac{i}{h}a$,
$h\in\left(0,1\right]$ of unitary connections on $L$. Denote the
corresponding tensor product connection on $S\otimes L$ by $\nabla^{S\otimes L}\coloneqq\nabla^{S}\otimes1+1\otimes\nabla^{h}$.
Each such connection defines a coupled Dirac operator 
\begin{align*}
D_{h}\coloneqq hD_{A_{0}}+ic\left(a\right)=hc\circ\left(\nabla^{S\otimes L}\right):C^{\infty}(X;S\otimes L)\rightarrow C^{\infty}(X;S\otimes L)
\end{align*}
for $h\in\left(0,1\right]$. The Dirac operator $D_{h}$ being elliptic
and self-adjoint has a discrete spectrum of eigenvalues with finite
multiplicities. 

The eta function of $D_{h}$ is defined by the formula
\begin{align}
\eta\left(D_{h},s\right)\coloneqq & \sum_{\begin{subarray}{l}
\quad\:\lambda\neq0\\
\lambda\in\textrm{Spec}\left(D_{h}\right)
\end{subarray}}\textrm{sign}(\lambda)|\lambda|^{-s}=\frac{1}{\Gamma\left(\frac{s+1}{2}\right)}\int_{0}^{\infty}t^{\frac{s-1}{2}}\textrm{tr}\left(D_{h}e^{-tD_{h}^{2}}\right)dt,\quad\forall s\in\mathbb{C}.\label{eq:eta invariant definition}
\end{align}
We shall use the convention that $\textrm{Spec}(D_{h})$ is a multiset
with each eigenvalue of $D_{h}$ being counted with its multiplicity.
From the Weyl law for elliptic operators, the above series is seen
to converge for $\textrm{Re}(s)>n.$ Further in \cite{APSI,APSIII}
it was shown that the eta function above \prettyref{eq:eta invariant definition}
has a meromorphic continuation to the entire complex $s$-plane and
further has no pole at zero. The eta invariant of the Dirac operator
$D_{h}$ is then defined to be the value of \prettyref{eq:eta invariant definition}
at zero
\begin{equation}
\eta_{h}\coloneqq\eta\left(D_{h},0\right).\label{eq:eta invariant}
\end{equation}
From \prettyref{eq:eta invariant definition} the above is seen formally
to be the signature of the Dirac operator, i.e. the difference between
the number of its positive and negative eigenvalues. A variant of
the above, known as the reduced eta invariant, is defined by including
the zero eigenvalue with an appropriate convention 
\begin{align*}
\bar{\eta}_{h}\coloneqq & \frac{1}{2}\left\{ k_{h}+\eta_{h}\right\} \\
k_{h}\coloneqq & \textrm{dim ker }\left(D_{h}\right).
\end{align*}

Much like the signature of a matrix, the eta invariant is left unchanged
under positive scaling
\begin{equation}
\eta\left(D_{h},0\right)=\eta\left(cD_{h},0\right);\quad\forall c>0.\label{eq: eta scale invariant}
\end{equation}
With
\[
K_{t,h}\left(x,x'\right)\coloneqq D_{h}e^{-tD_{h}^{2}}\left(x,x'\right)\in C^{\infty}\left(X\times X;S\otimes S^{*}\right)
\]
 being the Schwartz kernel of the given heat operator, defined with
respect to the Riemannian volume density, we denote by $\textrm{tr}\left(K_{t,h}\left(x,x\right)\right)$
its point-wise trace along the diagonal. A generalization of the eta
function \prettyref{eq:eta invariant definition} is then given by
\begin{align}
\eta\left(D_{h},s,x\right)= & \frac{1}{\Gamma\left(\frac{s+1}{2}\right)}\int_{0}^{\infty}t^{\frac{s-1}{2}}\textrm{tr}\left(K_{t,h}\left(x,x\right)\right)dt,\label{eq:eta function diagonal}
\end{align}
as a function of $s\in\mathbb{C},\,x\in X$. It was shown in \cite[Thm 2.6]{Bismut-Freed-II}
that the function $\eta\left(D_{h},s,x\right)$ is holomorphic in
$s$ and smooth in $x$ for $\textrm{Re}(s)>-2$. This from the above
\prettyref{eq:eta function diagonal} is clearly equivalent to 
\begin{align}
\textrm{tr}\left(K_{t,h}\right)= & O\left(t^{\frac{1}{2}}\right),\quad\textrm{as}\:t\rightarrow0.\label{eq:pointwise trace asymp as t->0}
\end{align}
The integral
\begin{equation}
\eta_{h}=\int_{0}^{\infty}\frac{1}{\sqrt{\pi t}}\textrm{tr}\left(D_{h}e^{-tD_{h}^{2}}\right)dt\label{eq: eta integral}
\end{equation}
is convergent with its value being the eta invariant \prettyref{eq:eta invariant}.

\subsection{\label{subsec:The-Semi-classical-calculus}The Semi-classical calculus}

Next we state some requisite facts from semi-classical analysis that
shall be used in the paper, \cite{GuilleminSternberg-Semiclassical,Zworski}
provide the standard references. For any $l\times l$ complex matrix
$A=\left(a_{ij}\right)\in\mathfrak{gl}\left(l\right)$, we denote
$\left|A\right|=\max_{ij}\left|a_{ij}\right|$. The symbol space $S^{m}\left(\mathbb{R}^{2n};\mathbb{C}^{l}\right)$
is defined as the space of maps $a:\left(0,1\right]_{h}\rightarrow C^{\infty}\left(\mathbb{R}_{x,\xi}^{2n};\mathfrak{gl}\left(l\right)\right)$
for which each semi-norm
\[
\left\Vert a\right\Vert _{\alpha,\beta}:=\text{sup}_{\substack{x,\xi}
,h}\langle\xi\rangle^{-m+|\beta|}\left|\partial_{x}^{\alpha}\partial_{\xi}^{\beta}a(x,\xi;h)\right|,\quad\alpha,\beta\in\mathbb{N}_{0}^{n},
\]
is finite. The more refined class $a\in S_{{\rm cl\,}}^{m}\left(\mathbb{R}^{2n};\mathbb{C}^{l}\right)$
of classical symbols consists of those for which there exists an $h$-independent
sequence $a_{k}$, $k=0,1,\ldots$ of symbols satisfying
\begin{equation}
a-\left(\sum_{k=0}^{N}h^{k}a_{k}\right)\in h^{N+1}S^{m}\left(\mathbb{R}^{2n};\mathbb{C}^{l}\right),\;\forall N.\label{eq:classical symbolic expansion}
\end{equation}
 Any given $a\in S^{m}\left(\mathbb{R}^{2n};\mathbb{C}^{l}\right),S_{{\rm cl\,}}^{m}\left(\mathbb{R}^{2n};\mathbb{C}^{l}\right)$
in one of the symbol classes above defines a one-parameter family
of operators $a^{W}\in\Psi^{m}\left(\mathbb{R}^{2n};\mathbb{C}^{l}\right),\Psi_{{\rm cl\,}}^{m}\left(\mathbb{R}^{2n};\mathbb{C}^{l}\right)$
via Weyl quantization whose Schwartz kernel is given by 
\[
a^{W}\coloneqq\frac{1}{\left(2\pi h\right)^{n}}\int e^{i\left(x-y\right).\xi/h}a\left(\frac{x+y}{2},\xi;h\right)d\xi.
\]
The above pseudodifferential classes of operators are closed under
the usual operations of composition and formal-adjoint. Furthermore
the classes are invariant under changes of coordinates and basis for
$\mathbb{C}^{l}$. Thus one may invariantly the classes of operators
$\Psi^{m}\left(X;E\right),\Psi_{{\rm cl\,}}^{m}\left(X;E\right)$
acting on $C^{\infty}\left(X;E\right)$ associated to any complex,
Hermitian vector bundle $\left(E,h^{E}\right)$ on a smooth compact
manifold $X$. 

The principal symbol of a classical pseudodifferential operator $A\in\Psi_{{\rm cl\,}}^{m}\left(X;E\right)$
is defined as an element in $\sigma\left(A\right)\in S^{m}\left(X;\textrm{End}\left(E\right)\right)\subset C^{\infty}\left(X;\textrm{End}\left(E\right)\right).$
It is given by $\sigma\left(A\right)=a_{0}$ the leading term in the
symbolic expansion \prettyref{eq:classical symbolic expansion} of
its full Weyl symbol. The principal symbol is multiplicative, commutes
with adjoints and fits into a symbol exact sequence 
\begin{align}
\sigma\left(AB\right) & =\sigma\left(A\right)\sigma\left(B\right)\nonumber \\
\sigma\left(A^{*}\right) & =\sigma\left(A\right)^{*}\nonumber \\
0\rightarrow h\Psi_{{\rm cl\,}}^{m}\left(X;E\right)\rightarrow & \Psi_{{\rm cl\,}}^{m}\left(X;E\right)\xrightarrow{\sigma}S^{m}\left(X;\textrm{End}\left(E\right)\right),\label{eq:properties of symbol}
\end{align}
where the formal adjoints above are defined with respect to the same
Hermitian metric $h^{E}$. The quantization map 
\begin{align}
\textrm{Op}:S^{m}\left(X;\textrm{End}\left(E\right)\right) & \rightarrow\Psi_{{\rm cl\,}}^{m}\left(X;E\right)\quad\textrm{ satisfying }\nonumber \\
\sigma\left(\textrm{Op}\left(a\right)\right) & =a\in S^{m}\left(X;\textrm{End}\left(E\right)\right)\label{eq:quantization map}
\end{align}
gives an inverse to the principal symbol map and we sometimes use
the alternate notation $\textrm{Op}\left(a\right)=a^{W}$. The quantization
map above is however non-canonical and depends on the choice of a
coordinate atlas, with local trivializations for $E$, as well as
a subordinate partition of unity. From the multiplicative property
of the symbol \prettyref{eq:properties of symbol}, it then follows
that $\left[a^{W},b^{W}\right]\in h\Psi_{{\rm cl\,}}^{m-1}\left(X;E\right)$
when $b\in S^{0}\left(X\right)$ is a scalar function. We shall then
define $H_{b}\left(a\right)\coloneqq\frac{i}{h}\sigma\left(\left[a^{W},b^{W}\right]\right)\in S^{m-1}\left(X;\textrm{End}\left(E\right)\right)$
however noting again that its definition depends on the quantization
scheme, and in particular the local trivializations used in defining
$\textrm{Op}$. It is given however by the Poisson bracket $H_{b}\left(a\right)=\left\{ a,b\right\} $
assuming that both sides are computed in the same defining trivialization.

Each $A\in\Psi_{{\rm cl\,}}^{m}\left(X;E\right)$ has a wavefront
set defined invariantly as a subset $WF\left(A\right)\subset\overline{T^{*}X}$
of the fibrewise radial compactification of the cotangent bundle $T^{*}X$.
It is locally defined as follows, $\left(x_{0},\xi_{0}\right)\notin WF\left(A\right)$,
$A=a^{W}$, if and only if there exists an open neighborhood $\left(x_{0},\xi_{0};0\right)\in U\subset\overline{T^{*}X}\times\left(0,1\right]_{h}$
such that $a\in h^{\infty}\left\langle \xi\right\rangle ^{-\infty}C^{k}\left(U;\mathbb{C}^{l}\right)$
for all $k$. The wavefront set satisfies the basic properties under
addition, multiplication and adjoints $WF\left(A+B\right)\subset WF\left(A\right)\cup WF\left(B\right)$,
$WF\left(AB\right)\subset WF\left(A\right)\cap WF\left(B\right)$
and $WF\left(A^{*}\right)=WF\left(A\right)$. The wavefront set $WF\left(A\right)=\emptyset$
is empty if and only if $A\in h^{\infty}\Psi^{-\infty}\left(X;E\right)$
while we say that two operators $A=B$ microlocally on $U\subset\overline{T^{*}X}$
if $WF\left(A-B\right)\cap U=\emptyset$. 

An operator $A\in\Psi_{{\rm cl\,}}^{m}\left(X;E\right)$ is said to
be elliptic if $\left\langle \xi\right\rangle ^{m}\sigma\left(A\right)^{-1}$
exists and is uniformly bounded on $T^{*}X$. If $A\in\Psi_{{\rm cl\,}}^{m}\left(X;E\right)$,
$m>0$, is formally self-adjoint, such that $A+i$ is elliptic, then
it is essentially self-adjoint (with domain $C_{c}^{\infty}\left(X;E\right)$)
as an unbounded operator on $L^{2}\left(X;E\right)$. Beals's lemma
further implies that its resolvent $\left(A-z\right)^{-1}\in\Psi_{{\rm cl\,}}^{-m}\left(X;E\right)$,
$z\in\mathbb{C}$, $\textrm{Im}z\neq0$, exists and is pseudo-differential.
The Helffer-Sjöstrand formula now expresses the function $f\left(A\right)$,
$f\in\mathcal{S}\left(\mathbb{R}\right)$, of such an operator in
terms of its resolvent 
\[
f\left(A\right)=\frac{1}{\pi}\int_{\mathbb{C}}\bar{\partial}\tilde{f}\left(z\right)\left(A-z\right)^{-1}dzd\bar{z},
\]
with $\tilde{f}$ denoting an almost analytic continuation of $f$.
One further has $WF\left(f\left(A\right)\right)\subset\Sigma_{\textrm{spt}\left(f\right)}^{A}\coloneqq\bigcup_{\lambda\in\textrm{spt}\left(f\right)}\Sigma_{\lambda}^{A}$
where 
\begin{equation}
\Sigma_{\lambda}^{A}=\left\{ \left(x,\xi\right)\in T^{*}X|\det\left(\sigma\left(A\right)\left(x,\xi\right)-\lambda I\right)=0\right\} .\label{eq:energy level}
\end{equation}
is classical $\lambda$-energy level of $A$. 

\subsubsection{\label{subsec:The-class-,}The class $\Psi_{\delta}^{m}\left(X\right)$}

We shall also need a more exotic class of scalar symbols $S_{\delta}^{m}\left(\mathbb{R}^{2n};\mathbb{C}\right)$
defined for each $0\leq\delta<\frac{1}{2}$. A function $a:\left(0,1\right]_{h}\rightarrow C^{\infty}\left(\mathbb{R}_{x,\xi}^{2n};\mathbb{C}\right)$
is said to be in this class if and only if 
\begin{equation}
\left\Vert a\right\Vert _{\alpha,\beta}:=\text{sup}_{\substack{x,\xi}
,h}h^{\left(\left|\alpha\right|+\left|\beta\right|\right)\delta}\left|\partial_{x}^{\alpha}\partial_{\xi}^{\beta}a(x,\xi;h)\right|\label{eq: delta pseudodifferential estimates}
\end{equation}
is finite $\forall\alpha,\beta\in\mathbb{N}_{0}^{n}.$ This class
of operators is also closed under the standard operations of composition,
adjoint and changes of coordinates; allowing for the definition of
the same exotic pseudo-differential algebra $\Psi_{\delta}^{m}\left(X\right)$
on a compact manifold. The class $S_{\delta}^{m}\left(X\right)$ is
a family of functions $a:\left(0,1\right]_{h}\rightarrow C^{\infty}\left(T^{*}X;\mathbb{C}\right)$
satisfying the estimates \prettyref{eq: delta pseudodifferential estimates}
in every coordinate chart and induced trivialization. Such a family
can be quantized to $a^{W}\in\Psi_{\delta}^{m}\left(X\right)$ satisfying
$a^{W}b^{W}=\left(ab\right)^{W}+h^{1-2\delta}\Psi_{\delta}^{m+m'-1}\left(X\right)$,
$\frac{i}{h^{1-2\delta}}\sigma\left(\left[a^{W},b^{W}\right]\right)=\left[\left\{ a,b\right\} \right]$
for another $b\in S_{\delta}^{m'}\left(X\right)$. The operators in
$\Psi_{\delta}^{0}\left(X\right)$ are uniformly bounded on $L^{2}\left(X\right)$.
Finally, the wavefront an operator $A\in\Psi_{\delta}^{m}\left(X;E\right)$
is similarly defined and satisfies the same basic properties as before.

\subsection{\label{subsec:Diophantine-approximation}Diophantine approximation}

We finally collect some requisite notions from Diophantine approximation.
These shall be useful later in \prettyref{subsec:Irrational-elliptic-flows}.
We refer to the texts \cite{Bugeaud2004,SchmidtW1980} for the background
and proofs of the statements below.

Let $a\in\mathbb{R}$ be a real number. Its irrationality exponent/measure
is defined by
\begin{align}
\mu\left(a\right) & \coloneqq\inf\left\{ \mu|\left|a-\frac{p}{q}\right|<\frac{1}{q^{\mu}},\,\textrm{has finitely many rational solutions }\frac{p}{q}\in\mathbb{Q}\right\} \label{eq:irrationality exp def 1}\\
 & =\inf\left\{ \mu|\exists C>0\textrm{ s.t. }\left|a-\frac{p}{q}\right|>\frac{C}{q^{\mu}},\,\forall\frac{p}{q}\in\mathbb{Q}\setminus\left\{ a\right\} \right\} \label{eq:irrationality exp def 2}
\end{align}
 where we set $\mu\left(a\right)=\infty$ when the sets above are
empty. 

It is easy to check that $\mu\left(a\right)=1$ for $a\in\mathbb{Q}$
rational. A theorem of Dirichlet shows that $\mu\left(a\right)\geq2$
for $a$ irrational. The map $\mu:\mathbb{R\setminus Q}\rightarrow\left[2,\infty\right)$
is known to be surjective while $\mu\left(a\right)=2$ for almost
all reals with respect to the Lebesgue measure. A number $a$ with
$\mu\left(a\right)=2$ and for which the infimum in \prettyref{eq:irrationality exp def 2}
is attained is called badly approximable. Roth's theorem shows that
$\mu\left(a\right)=2$ for irrational algebraic integers, it is conjectured
however that no such (of degree at least 3) is badly approximable.
Furthermore conversely there are transcendental $a$ with $\mu\left(a\right)=2$,
Euler's number $e$ being such an example. The reals $a$ for which
$\mu\left(a\right)=\infty$ are called a Liouville numbers, these
form a dense albeit Lebesgue measure zero subset of the reals.

A generalization of the above, the irrationality exponent of simultaneous
Diophantine approximation, can be defined for a tuple of real numbers
$\left(a_{1},a_{2},\ldots,a_{n}\right)\in\mathbb{R}^{n}\setminus\left\{ 0\right\} $,
$n\geq2$, via 

\begin{equation}
\nu\left(a_{1},\ldots,a_{n}\right)\coloneqq\inf\left\{ \nu|\exists C>0\textrm{ s.t. }d\left(\left(ta_{1},\ldots,ta_{n}\right);\mathbb{Z}^{n}\right)>Ct^{1-\nu},\,\forall\left(ta_{1},\ldots,ta_{n}\right)\notin\mathbb{Z}^{n}\right\} ,\label{eq:index tuple}
\end{equation}
where $d$ above denotes the distance from the standard lattice $\mathbb{Z}^{n}\subset\mathbb{R}^{n}$.
The above \prettyref{eq:index tuple} is seen to be related to the
exponent \prettyref{eq:irrationality exp def 2} via $\mu\left(a\right)=\nu\left(1,a\right)$
and is scale invariant $\nu\left(ca_{1},ca_{2}\right)=\nu\left(a_{1},a_{2}\right)$,
$c\neq0$. Further it is easy to check that $\nu\left(a_{1},\ldots,a_{n}\right)\leq\min_{i\neq j}\nu\left(a_{i},a_{j}\right)$
and that $\nu\left(1,a_{2},\ldots,a_{n}\right)=1$ for $\left(a_{2},\ldots,a_{n}\right)\in\mathbb{Q}^{n-1}$
. The higher dimensional analogue of Dirichlet's theorem says $1+\frac{1}{n-1}\leq\nu\left(1,a_{2},\ldots,a_{n}\right)$
for $\left(a_{2},\ldots,a_{n}\right)\notin\mathbb{Q}^{n-1}$ with
again the equality holding for almost all tuples. The higher dimensional
analogue of Roth's theorem is due to Schmidt: if $a_{2},\ldots,a_{n}$
are algebraic integers such that $\left\{ 1,a_{2},\ldots,a_{n}\right\} $
are rationally independent then $\nu\left(1,a_{2},\ldots,a_{n}\right)=1+\frac{1}{n-1}$.

\section{\label{sec:Microlocal-trace-expansion}Microlocal trace expansion}

The Dirac operator $D_{h}$ \prettyref{eq:Semiclassical Magnetic Dirac}
has principal symbol and characteristic variety 
\begin{align}
\sigma\left(D_{h}\right)\left(x,\xi\right) & =c\left(\xi+a\right)\in C^{\infty}\left(T^{*}X;\textrm{End}\left(S\right)\right)\label{eq:principal symbol}\\
\Sigma & \coloneqq\left\{ \left(x,\xi\right)|\sigma\left(D_{h}\right)\left(x,\xi\right)=0\right\} \nonumber \\
 & =\left\{ \left(x,\xi\right)|\xi=-a\left(x\right)\right\} \label{eq:characteristic variety}
\end{align}
given by Clifford multiplication and the graph of the one form $a$
respectively.

In \cite[Sec. 7]{Savale2017-Koszul} an on diagonal expansion for
functions $\phi\left(\frac{D_{h}}{\sqrt{h}}\right)$, $\phi\in\mathcal{S}\left(\mathbb{R}\right)$,
of the Dirac operator was proved. Namely we showed the existence of
tempered distributions 
\[
U_{j,p}\left(s\right)\in C^{\infty}\left(X;S\otimes L\otimes\mathcal{S}'\left(\mathbb{R}_{s}\right)\right),\quad
\]
$j\in\mathbb{N}_{0},\,p\in X,$ such that
\begin{equation}
\phi\left(\frac{D_{h}}{\sqrt{h}}\right)\left(p,p\right)=h^{-n/2}\left(\sum_{j=0}^{N}U_{j,p}\left(\phi\right)h^{j/2}\right)+h^{\left(N+1-n\right)/2}O\left(\sum_{k=0}^{n+1}\left\Vert \left\langle \xi\right\rangle ^{N}\hat{\phi}^{\left(k\right)}\right\Vert _{L^{1}}\right)\label{eq: local on diagonal expansion}
\end{equation}
$\forall\phi\in\mathcal{S}\left(\mathbb{R}_{s}\right),p\in X,\,N\in\mathbb{N}$.
Here we show a further microlocal version of this result, considering
functions of $B\phi\left(\frac{D_{h}}{\sqrt{h}}\right)$, $\phi\in\mathcal{S}\left(\mathbb{R}\right)$,
$B\in\Psi_{{\rm cl\,}}^{0}\left(X;S\otimes L\right)$. The leading
part of this expansion shall be shown to concentrate on the characteristic
variety $\Sigma$ \prettyref{eq:characteristic variety}.

We first fix some terminology. Fixing a point $p\in X$ there is an
orthonormal basis $e_{0,p}=\frac{R}{\left|R\right|}$,$\left\{ e_{j,p},\,e_{j+m,p}\right\} _{j=1}^{m}\in R^{\perp}$,
of the tangent space at $p$ consisting of eigenvectors of $\mathfrak{J}_{p}$
with eigenvalues $0,\pm i\mu_{j}$, $j=1,\ldots,m$, such that 
\begin{equation}
da\left(p\right)=\sum_{j=1}^{m}\mu_{j}e_{j,p}^{*}\wedge e_{j+m,p}^{*}.\label{eq: da diagonal form}
\end{equation}
Using the parallel transport from this basis, fix a geodesic coordinate
system $\left(x_{0},\ldots,x_{2m}\right)$ on an open neighborhood
of $p\in\Omega$. Let $e_{j}=w_{j}^{k}\partial_{x_{k}}$, $0\leq j\leq2m$,
be the local orthonormal frame of $TX$ obtained by parallel transport
of $e_{j,p}=\left.\partial_{x_{j}}\right|_{p}$, $0\leq j\leq2m$,
along geodesics. We then have 
\begin{align*}
w_{j}^{k}g_{kl}w_{r}^{l} & =\delta_{jr},\\
\left.w_{j}^{k}\right|_{p} & =\delta_{j}^{k},
\end{align*}
with $g_{kl}$ being the components of the metric in these coordinates.
Choose an orthonormal basis $\left\{ s_{j,p}\right\} _{j=1}^{2^{m}}$for
$S_{p}$ in which Clifford multiplication
\begin{equation}
\left.c\left(e_{j}\right)\right|_{p}=\gamma_{j}\label{eq: Clifford multiplication standard}
\end{equation}
is standard. Choose an orthonormal basis $\mathtt{l}_{p}$ for $L_{p}$.
Parallel transport the bases $\left\{ s_{j,p}\right\} _{j=1}^{2^{m}}$,
$\mathtt{l}_{p}$ along geodesics using the spin connection $\nabla^{S}$
and unitary family of connections $\nabla^{h}=A_{0}+\frac{i}{h}a$
to obtain trivializations $\left\{ s_{j}\right\} _{j=1}^{2^{m}}$,
$\mathtt{l}$ of $S$, $L$ on $\Omega$. Since Clifford multiplication
is parallel, the relation \prettyref{eq: Clifford multiplication standard}
now holds on $\Omega$. The connection $\nabla^{S\otimes L}=\nabla^{S}\otimes1+1\otimes\nabla^{h}$
can be expressed in this frame and these coordinates as
\begin{equation}
\nabla^{S\otimes L}=d+A_{j}^{h}dx^{j}+\Gamma_{j}dx^{j},
\end{equation}
 where each $A_{j}^{h}$ is a Christoffel symbol of $\nabla^{h}$
and each $\Gamma_{j}$ is a Christoffel symbol of the spin connection
$\nabla^{S}$. Since the section $\mathtt{l}$ is obtained via parallel
transport along geodesics, the connection coefficient $A_{j}^{h}$
can be written in terms of the curvature $F_{jk}^{h}dx^{j}\wedge dx^{k}$
of $\nabla^{h}$ 
\begin{equation}
A_{j}^{h}(x)=\int_{0}^{1}d\rho\left(\rho x^{k}F_{jk}^{h}\left(\rho x\right)\right).
\end{equation}
The dependence of the curvature coefficients $F_{jk}^{h}$ on the
parameter $h$ is seen to be linear in $\frac{1}{h}$ via 
\begin{equation}
F_{jk}^{h}=F_{jk}^{0}+\frac{i}{h}\left(da\right){}_{jk}\label{eq:curvature linear in 1/h}
\end{equation}
despite the fact that they are expressed in the $h$ dependent frame
$\mathtt{l}$. This is because a gauge transformation from an $h$
independent frame $\mathtt{l}_{0}$ into $\mathtt{l}$ changes the
curvature coefficient by conjugation. Since $L$ is a line bundle
this is conjugation by a function and hence does not change the coefficient.
Furthermore, the coefficients in the Taylor expansion of \prettyref{eq:curvature linear in 1/h}
at $0$ can be expressed in terms of the covariant derivatives $\left(\nabla^{A_{0}}\right)^{l}F_{jk}^{0},$
$\left(\nabla^{A_{0}}\right)^{l}\left(da\right){}_{jk}$ evaluated
at $p$. 

The gauge transformation relating the $h$-dependent frame $\mathtt{l}$
with the $h$-independent frame $\mathtt{l}_{0}$ is given by 
\begin{align}
\mathtt{l} & =c\exp\left\{ \underbrace{-\int_{0}^{1}\rho\,d\rho x^{j}A_{0,j}\left(\rho x\right)-\frac{i}{h}\underbrace{\int_{0}^{1}\rho\,d\rho x^{j}a_{j}\left(\rho x\right)}_{\coloneqq\varphi}}_{\varphi_{h}}\right\} \mathtt{l}_{0}\nonumber \\
\varphi & =x^{j}a_{j}\left(0\right)+O\left(x^{2}\right)\label{eq:relation two frames}
\end{align}
where $A_{0}+\frac{i}{h}a$ denotes the connection form for $\nabla^{h}$
in the $\mathtt{l}_{0}$ trivialization while $c>0$ is a constant
which can be taken to be $1$ by an appropriate choice of $\mathtt{l}_{0}$.

Next, using the Taylor expansion
\begin{equation}
\left(da\right){}_{jk}=\left(da\right){}_{jk}\left(0\right)+x^{l}a_{jkl},\label{eq: Taylor expansion da}
\end{equation}
we see that the connection $\nabla^{S\otimes L}$ has the form
\begin{equation}
\nabla^{S\otimes L}=d+\left[\frac{i}{h}\left(\frac{x^{k}}{2}\left(da\right){}_{jk}\left(0\right)+x^{k}x^{l}A_{jkl}\right)+x^{k}A_{jk}^{0}+\Gamma_{j}\right]dx^{j}\label{eq: connection in geodesic}
\end{equation}
where 
\begin{eqnarray*}
A_{jk}^{0} & = & \int_{0}^{1}d\rho\left(\rho F_{jk}^{0}\left(\rho x\right)\right)\\
A_{jkl} & = & \int_{0}^{1}d\rho\left(\rho a_{jkl}\left(\rho x\right)\right)
\end{eqnarray*}
and $\Gamma_{j}$ are all independent of $h$. Finally from \prettyref{eq: Clifford multiplication standard}
and \prettyref{eq: connection in geodesic} we may write down the
expression for the Dirac operator \prettyref{eq:Semiclassical Magnetic Dirac},
given as $D_{h}=hc\circ\left(\nabla^{S\otimes L}\right)$, in terms
of the chosen frame and coordinates to be 
\begin{align}
D_{h} & =\gamma^{r}w_{r}^{j}\left[h\partial_{x_{j}}+i\frac{x^{k}}{2}\left(da\right){}_{jk}\left(0\right)+ix^{k}x^{l}A_{jkl}+h\left(x^{k}A_{jk}^{0}+\Gamma_{j}\right)\right]\label{eq: Dirac operator geodesic coordinates}\\
 & =\gamma^{r}\left[w_{r}^{j}h\partial_{x_{j}}+iw_{r}^{j}\frac{x^{k}}{2}\left(da\right){}_{jk}\left(0\right)+\frac{1}{2}hg^{-\frac{1}{2}}\partial_{x_{j}}\left(g^{\frac{1}{2}}w_{r}^{j}\right)\right]+\\
 & \gamma^{r}\left[iw_{r}^{j}x^{k}x^{l}A_{jkl}+hw_{r}^{j}\left(x^{k}A_{jk}^{0}+\Gamma_{j}\right)-\frac{1}{2}hg^{-\frac{1}{2}}\partial_{x_{j}}\left(g^{\frac{1}{2}}w_{r}^{j}\right)\right]\in\Psi_{{\rm cl\,}}^{1}\left(\Omega_{s}^{0};\mathbb{C}^{2^{m}}\right)\nonumber 
\end{align}
In the second expression above, both square brackets are self-adjoint
with respect to the Riemannian density $e^{1}\wedge\ldots\wedge e^{n}=\sqrt{g}dx\coloneqq\sqrt{g}dx^{1}\wedge\ldots\wedge dx^{n}$,
where $g=\det\left(g_{ij}\right)$. Again one may obtain an expression
self-adjoint with respect to the Euclidean density $dx$ in the framing
$g^{\frac{1}{4}}u_{j}\otimes\mathtt{l},1\leq j\leq2^{m}$. The result
being an addition of the term $h\gamma^{j}w_{j}^{k}g^{-\frac{1}{4}}\left(\partial_{x_{k}}g^{\frac{1}{4}}\right)$. 

Let $i_{g}$ be the injectivity radius of $g^{TX}$ . Define the cutoff
$\chi\in C_{c}^{\infty}\left(-1,1\right)$ such that $\chi=1$ on
$\left(-\frac{1}{2},\frac{1}{2}\right)$. We now modify the functions
$w_{j}^{k}$, outside the ball $B_{i_{g}/2}\left(p\right)$, such
that $w_{j}^{k}=\delta_{j}^{k}$ (and hence $g_{jk}=\delta_{jk\mathtt{l}_{0}}$)
are standard outside the ball $B_{i_{g}}\left(p\right)$ of radius
$i_{g}$ centered at $p$. This again gives 
\begin{align}
\mathbb{D} & =\gamma^{r}\left[w_{r}^{j}h\partial_{x_{j}}+iw_{r}^{j}\frac{x^{k}}{2}\left(da\right){}_{jk}\left(0\right)+\frac{1}{2}hg^{-\frac{1}{2}}\partial_{x_{j}}\left(g^{\frac{1}{2}}w_{r}^{j}\right)\right]+\label{eq: Localized Dirac}\\
 & \chi\left(\left|x\right|/i_{g}\right)\gamma^{r}\left[iw_{r}^{j}x^{k}x^{l}A_{jkl}+hw_{r}^{j}\left(x^{k}A_{jk}^{0}+\Gamma_{j}\right)-\frac{1}{2}hg^{-\frac{1}{2}}\partial_{x_{j}}\left(g^{\frac{1}{2}}w_{r}^{j}\right)\right]\nonumber \\
 & \qquad\qquad\qquad\qquad\qquad\qquad\qquad\qquad\qquad\in\Psi_{{\rm cl\,}}^{1}\left(\mathbb{R}^{n};\mathbb{C}^{2^{m}}\right)\nonumber 
\end{align}
as a well defined operator on $\mathbb{R}^{n}$ formally self adjoint
with respect to $\sqrt{g}dx$. Again $\mathbb{D}+i$ being elliptic
in the class $S^{0}\left(m\right)$ for the order function 
\[
m=\sqrt{1+g^{jl}\left(\xi_{j}+\frac{x^{k}}{2}\left(da\right)_{jk}\left(0\right)\right)\left(\xi_{l}+\frac{x^{r}}{2}\left(da\right)_{lr}\left(0\right)\right)},
\]
the operator $\mathbb{\mathbb{D}}$ is essentially self adjoint. 

Letting $H\left(s\right)\in\mathcal{S}'\left(\mathbb{R}_{s}\right)$
denote the Heaviside distribution, below we  define the following
elementary tempered distributions 
\begin{align}
v_{a;p}\left(s\right) & \coloneqq s^{a},\;a\in\mathbb{N}_{0}\label{eq: elementary distribution 1}\\
v_{a,b,c,\varLambda;p}\left(s\right) & \coloneqq\partial_{s}^{a}\left[\left|s\right|s^{b}\left(s^{2}-2\varLambda\right)^{c-\frac{1}{2}}H\left(s^{2}-2\varLambda\right)\right],\label{eq: elementary distribution 2}\\
 & \;\qquad\qquad\qquad\qquad\left(a,b,c;\varLambda\right)\in\mathbb{N}_{0}\times\mathbb{Z}\times\mathbb{N}_{0}\times\mu.\left(\mathbb{N}_{0}^{m}\setminus0\right)\nonumber 
\end{align}
as in \cite[Sec. 7]{Savale2017-Koszul}.

We now have the following. 
\begin{thm}
\label{thm: microlocal trace expansion} Let $B\in\Psi_{{\rm cl\,}}^{0}\left(X;S\otimes L\right)$
be a classical pseudodifferential operator. 

There exist tempered distributions $U_{B,j,x,x'}\in C^{\infty}\left(TX\oplus TX;\textrm{End}\left(S\otimes L\right)\otimes\mathcal{S}'\left(\mathbb{R}_{s}\right)\right)$,
$j=0,1,2,\ldots$, such that one has the following off-diagonal expansion
for the Schwartz kernel
\begin{align}
B\phi\left(\frac{D_{h}}{\sqrt{h}}\right)\left(x_{h},x'_{h}\right) & =h^{-n/2}\left(\sum_{j=0}^{N}e^{\varphi_{h}\left(x\right)}U_{B,j,x,x'}\left(\phi\right)e^{-\varphi_{h}\left(x'\right)}h^{j/2}\right)\label{eq: microlocal trace expansion}\\
 & \qquad\qquad\qquad+h^{\left(N+1-n\right)/2}O\left(\sum_{k=0}^{n+1}\left\Vert \left\langle \xi\right\rangle ^{2N}\hat{\phi}^{\left(k\right)}\right\Vert _{L^{1}}\right),\\
\textrm{where }\quad\left(x_{h},x'_{h}\right) & \coloneqq\left(\exp_{x}\left(\sqrt{h}x\right),\exp_{x}\left(\sqrt{h}x'\right)\right),\label{eq:off diagonal points}
\end{align}
for each $N\in\mathbb{N}$, $x,x'\in T_{p}X$, $p\in X$ and $\phi\in\mathcal{S}\left(\mathbb{R}_{s}\right)$. 

Each coefficient of the expansion above can be written in terms of
the distributions \prettyref{eq: elementary distribution 1}, \prettyref{eq: elementary distribution 2}
\begin{equation}
U_{B,j,x,x'}\left(s\right)=\sum_{a\leq2j+2}c_{B,j;a}\left(x,x'\right)s^{a}+\sum_{\begin{subarray}{l}
\varLambda\in\mu.\left(\mathbb{N}_{0}^{m}\setminus0\right).\\
a,\left|b\right|,c\leq4j+4
\end{subarray}}c_{B,j;a,b,c,\varLambda}\left(x,x'\right)v_{a,b,c,\varLambda;p}\left(s\right),\label{eq: trace distribution structure}
\end{equation}
for some $h-$independent sections $c_{B,j;a},c_{B,j;a,b,c,\varLambda}\in C^{\infty}\left(TX\oplus TX;\textrm{End}\left(S\otimes L\right)\right)$. 

Moreover, the leading coefficient at the origin is given by 
\begin{equation}
U_{B,0,0,0}=\left(\left.b_{0}\right|_{\Sigma}\right).U_{0,p}\label{eq:leading coefficient}
\end{equation}
in terms of the leading coefficient of \prettyref{eq: local on diagonal expansion}
and the principal symbol $b_{0}=\sigma\left(B\right)$.
\end{thm}

\begin{proof}
We begin by writing $\phi=\phi_{0}+\phi_{1}$, with 
\begin{eqnarray*}
\phi_{0}\left(s\right) & = & \frac{1}{2\pi}\int_{\mathbb{R}}e^{i\xi s}\hat{\phi}\left(\xi\right)\chi\left(\frac{2\xi\sqrt{h}}{i_{g}}\right)d\xi\\
\phi_{1}\left(s\right) & = & \frac{1}{2\pi}\int_{\mathbb{R}}e^{i\xi s}\hat{\phi}\left(\xi\right)\left[1-\chi\left(\frac{2\xi\sqrt{h}}{i_{g}}\right)\right]d\xi
\end{eqnarray*}
given by Fourier inversion. 

First considering $\phi_{1}$, integration by parts gives the estimate
\[
\left|s^{n+1}\phi_{1}\left(s\right)\right|\leq C_{N}h^{\frac{N-1}{2}}\left(\sum_{k=0}^{n+1}\left\Vert \xi^{2N}\hat{\phi}^{\left(k\right)}\right\Vert _{L^{1}}\right),
\]
$\forall N\in\mathbb{N}$. Hence,
\[
\left\Vert D_{h}^{n+1-a}B\phi_{1}\left(\frac{D_{h}}{\sqrt{h}}\right)D_{h}^{a}\right\Vert _{L^{2}\rightarrow L^{2}}\leq C_{N}h^{\frac{n+N}{2}}\left(\sum_{k=0}^{n+1}\left\Vert \xi^{2N}\hat{\phi}^{\left(k\right)}\right\Vert _{L^{1}}\right),
\]
$\forall N\in\mathbb{N},\:\forall a=0,\ldots,n+1$. Semi-classical
elliptic estimate and Sobolev's inequality now give the estimate 
\begin{equation}
\left|B\phi_{1}\left(\frac{D_{h}}{\sqrt{h}}\right)\right|_{C^{0}\left(X\times X\right)}\leq C_{N}h^{\frac{n+N}{2}}\left(\sum_{k=0}^{n+1}\left\Vert \xi^{2N}\hat{\phi}^{\left(k\right)}\right\Vert _{L^{1}}\right)\label{eq: Schw ker localizes}
\end{equation}
$\forall N\in\mathbb{N}$, on the Schwartz kernel.

Next, considering $\phi_{0}$, we first use the change of variables
$\alpha=\xi\sqrt{h}$ to write 
\[
\phi_{0}\left(\frac{D_{h}}{\sqrt{h}}\right)=\frac{1}{2\pi\sqrt{h}}\int_{\mathbb{R}}e^{i\alpha\left(D_{A_{0}}+ih^{-1}c\left(a\right)\right)}\hat{\phi}\left(\frac{\alpha}{\sqrt{h}}\right)\chi\left(\frac{2\alpha}{i_{g}}\right)d\alpha.
\]
Now since $D_{h}=\mathbb{D}$ on $B_{i_{g}/2}\left(p\right)$, we
may use the finite propagation speed of the wave operators $e^{i\alpha h^{-1}D_{h}}$,
$e^{i\alpha h^{-1}\mathbb{D}}$ and microlocality of $B\in\Psi_{{\rm cl\,}}^{0}\left(X\right)$
to conclude 
\begin{equation}
B\phi_{0}\left(\frac{D_{h}}{\sqrt{h}}\right)\left(\sqrt{h}x,\sqrt{h}x'\right)=B\phi_{0}\left(\frac{\mathbb{D}}{\sqrt{h}}\right)\left(\sqrt{h}x,\sqrt{h}x'\right)\label{eq: finite propagation}
\end{equation}
for $h\ll1$. The right hand side above is defined using functional
calculus of self-adjoint operators, with standard local elliptic regularity
arguments implying the smoothness of its Schwartz kernel. By virtue
of \prettyref{eq: Schw ker localizes}, a similar estimate for $B\phi_{1}\left(\frac{\mathbb{D}}{\sqrt{h}}\right)$,
and \prettyref{eq: finite propagation} it now suffices to consider
$B\phi\left(\frac{\mathbb{D}}{\sqrt{h}}\right)$.

We now introduce the rescaling operator $\mathscr{R}:C^{\infty}\left(\mathbb{R}^{n};\mathbb{C}^{2^{m}}\right)\rightarrow C^{\infty}\left(\mathbb{R}^{n};\mathbb{C}^{2^{m}}\right)$,
$\left(\mathscr{R}s\right)\left(x\right)\coloneqq s\left(\frac{x}{\sqrt{h}}\right)$.
Conjugation by $\mathscr{R}$ amounts to the rescaling of coordinates
$x\rightarrow x\sqrt{h}$. A Taylor expansion in \prettyref{eq: Localized Dirac}
now gives the existence of classical ($h$-independent) self-adjoint,
first-order differential operators $\mathrm{\mathtt{D}}_{j}=a_{j}^{k}\left(x\right)\partial_{x_{k}}+b_{j}\left(x\right)$,
$j=0,1\ldots$, with polynomial coefficients (of degree at most $j+1$)
as well as $h$-dependent self-adjoint, first-order differential operators
$\mathrm{E}_{j}=\sum_{\left|\alpha\right|=N+1}x^{\alpha}\left[c_{j,\alpha}^{k}\left(x;h\right)\partial_{x_{k}}+d_{j,\alpha}\left(x;h\right)\right]$,
$j=0,1\ldots$, with uniformly $C^{\infty}$ bounded coefficients
$c_{j,\alpha}^{k},\,d_{j,\alpha}$ such that 
\begin{eqnarray}
\mathscr{R}\mathbb{D}\mathscr{R}^{-1} & = & \sqrt{h}\mathrm{\mathtt{D}}\quad\textrm{ with}\label{eq: rescaled Dirac}\\
\mathrm{\mathtt{D}} & = & \left(\sum_{j=0}^{N}h^{j/2}\mathrm{\mathtt{D}}_{j}\right)+h^{\left(N+1\right)/2}\mathrm{E}_{N+1},\;\forall N.\label{eq: Taylor expansion Dirac}
\end{eqnarray}
The coefficients of the polynomials $a_{j}^{k}\left(x\right),\,b_{j}\left(x\right)$
again involve the covariant derivatives of the curvatures $F^{TX},F^{A_{0}}$
and $da$ evaluated at $p$. Furthermore, the leading term in \prettyref{eq: Taylor expansion Dirac}
is easily computed 
\begin{align}
\mathrm{\mathtt{D}}_{0} & =\gamma^{j}\left[\partial_{x_{j}}+i\frac{x^{k}}{2}\left(da\right){}_{jk}\left(0\right)\right]\label{eq: leading term rescaled Dirac}\\
 & =\gamma^{0}\partial_{x_{0}}+\underbrace{\gamma^{j}\left[\partial_{x_{j}}+\frac{i\mu_{j}\left(p\right)}{2}x_{j+m}\right]+\gamma^{j+m}\left[\partial_{x_{j+m}}-\frac{i\mu_{j}\left(p\right)}{2}x_{j}\right]}_{\coloneqq\mathrm{\mathtt{D}}_{00}}\label{eq: leading term rescaled Dirac-1}
\end{align}
using \prettyref{eq: da diagonal form}, \prettyref{eq: Taylor expansion da}.
It is now clear from \prettyref{eq: rescaled Dirac} that 
\begin{equation}
\phi\left(\frac{\mathbb{D}}{\sqrt{h}}\right)\left(x,x'\right)=h^{-n/2}\phi\left(\mathrm{\mathtt{D}}\right)\left(\frac{x}{\sqrt{h}},\frac{x'}{\sqrt{h}}\right).\label{eq: rescaling Schw kernel}
\end{equation}
Next, let $I_{j}=\left\{ k=\left(k_{0},k_{1},\ldots\right)|k_{\alpha}\in\mathbb{N},\:\sum k_{\alpha}=j\right\} $
denote the set of partitions of the integer $j$ and set 
\begin{equation}
\mathtt{C}_{j}^{z}=\sum_{k\in I_{j}}\left(z-\mathrm{\mathtt{D}}_{0}\right)^{-1}\left[\Pi_{\alpha}\left[\mathrm{\mathtt{D}}_{k_{\alpha}}\left(z-\mathrm{\mathtt{D}}_{0}\right)^{-1}\right]\right].\label{eq: jth term kernel expansion}
\end{equation}
Local elliptic regularity estimates again give 
\begin{align*}
\left(z-\mathrm{\mathtt{D}}\right)^{-1} & =O_{L_{\textrm{loc}}^{2}\rightarrow L_{\textrm{loc}}^{2}}\left(\left|\textrm{Im}z\right|^{-1}\right)\quad\textrm{ and }\\
\mathtt{C}_{j}^{z} & =O_{L_{\textrm{loc}}^{2}\rightarrow L_{\textrm{loc}}^{2}}\left(\left|\textrm{Im}z\right|^{-2j-2}\right),
\end{align*}
$j=0,1,\ldots$. A straightforward computation using \prettyref{eq: Taylor expansion Dirac}
then yields 
\begin{equation}
\left(z-\mathrm{\mathtt{D}}\right)^{-1}-\left(\sum_{j=0}^{N}h^{j/2}\mathtt{C}_{j}^{z}\right)=O_{L_{\textrm{loc}}^{2}\rightarrow L_{\textrm{loc}}^{2}}\left(\left(\left|\textrm{Im}z\right|^{-2}h^{\frac{1}{2}}\right)^{N+1}\right).\label{eq: resolvent expansion}
\end{equation}
A similar expansion as \prettyref{eq: Taylor expansion Dirac} for
the operator $\left(1+\mathrm{\mathtt{\mathrm{\mathtt{D}}}}^{2}\right)^{\left(n+1\right)/2}\left(z-\mathrm{\mathrm{\mathtt{D}}}\right)$
also gives the bounds 
\begin{equation}
\left(1+\mathrm{\mathtt{\mathrm{\mathtt{D}}}}^{2}\right)^{-\left(n+1\right)/2}\left(z-\mathrm{\mathrm{\mathtt{D}}}\right)^{-1}-\left(\sum_{j=0}^{N}h^{j/2}\mathtt{C}_{j,n+1}^{z}\right)=O_{H_{\textrm{loc}}^{s}\rightarrow H_{\textrm{loc}}^{s+n+1}}\left(\left(\left|\textrm{Im}z\right|^{-2}h^{\frac{1}{2}}\right)^{N+1}\right)\label{eq: rescaled resolvent expansion}
\end{equation}
$\forall s\in\mathbb{R}$, for classical ($h$-independent) Sobolev
spaces $H_{\textrm{loc}}^{s}$. Here each $\mathtt{C}_{j,n+1}^{z}=O_{H_{\textrm{loc}}^{s}\rightarrow H_{\textrm{loc}}^{s+n+1}}\left(\left|\textrm{Im}z\right|^{-2j-2}\right)$
with the leading term being
\[
\mathtt{C}_{0,n+1}^{z}=\left(1+\mathrm{\mathrm{\mathtt{D}}}_{0}^{2}\right)^{-\left(n+1\right)/2}\left(z-\mathrm{\mathrm{\mathtt{D}}}_{0}\right)^{-1}.
\]
Finally, plugging the expansion \prettyref{eq: rescaled resolvent expansion}
into the Helffer-Sjöstrand formula 
\[
\phi\left(\mathrm{\mathrm{\mathtt{D}}}\right)=-\frac{1}{\pi}\int_{\mathbb{C}}\bar{\partial}\tilde{\varrho}\left(z\right)\left(1+\mathrm{\mathrm{\mathtt{D}}}^{2}\right)^{-\left(n+1\right)/2}\left(z-\mathrm{\mathtt{\mathrm{\mathtt{D}}}}\right)^{-1}dzd\bar{z},
\]
with $\varrho\left(s\right)\coloneqq\left\langle s\right\rangle ^{n+1}\phi\left(s\right)$,
gives 
\begin{equation}
\phi\left(\mathrm{\mathtt{D}}\right)\left(x,x'\right)=\left(\sum_{j=0}^{N}h^{j/2}U_{j,p}\left(\phi\right)\left(x,x'\right)\right)+h^{\left(N+1\right)/2}O\left(\sum_{k=0}^{n+1}\left\Vert \left\langle \xi\right\rangle ^{N}\hat{\phi}^{\left(k\right)}\right\Vert _{L^{1}}\right),\label{eq: Diagonal kernel expansion}
\end{equation}
$\forall y\in\mathbb{R}^{n}$, using Sobolev's inequality. Here each
\begin{equation}
U_{j,p}\left(\phi\right)\left(x,x'\right)=-\frac{1}{\pi}\int_{\mathbb{C}}\bar{\partial}\tilde{\varrho}\left(z\right)\mathtt{C}_{j,n+1}^{z}\left(x,x'\right)dzd\bar{z}\in\textrm{End}S_{p}^{TX}\label{eq: jth coefficient functional kernal}
\end{equation}
defines a smooth family (in $p\in X$) of distributions $U_{j}$ and
the remainder term in \prettyref{eq: Diagonal kernel expansion} comes
from the estimate $\bar{\partial}\tilde{\varrho}=O\left(\left|\textrm{Im}z\right|^{2N+2}\sum_{k=0}^{n+1}\left\Vert \left\langle \xi\right\rangle ^{2N}\hat{\phi}^{\left(k\right)}\right\Vert _{L^{1}}\right)$
on the almost analytic continuation (cf. \cite{Zworski} Sec. 3.1).
Substituting \prettyref{eq: Diagonal kernel expansion} into \prettyref{eq: rescaling Schw kernel}
gives the off-diagonal expansion for the Schwartz kernel of $\phi\left(\frac{D_{h}}{\sqrt{h}}\right)$.
However we importantly note that the expansion is expressed in the
$h-$dependent frame $\mathtt{l}$ obtained by parallel transport.
That is to say, using \prettyref{eq:relation two frames} to change
into the $h$ independent frame $\mathtt{l}_{0}$, that it is an expansion
for 
\begin{align}
K_{D,\phi}^{h}\left(x,x'\right) & \coloneqq e^{-\varphi_{h}\left(x\sqrt{h}\right)}\phi\left(\frac{D_{h}}{\sqrt{h}}\right)\left(x\sqrt{h},x'\sqrt{h}\right)e^{\varphi_{h}\left(x'\sqrt{h}\right)}.\nonumber \\
 & =h^{-n/2}\sum_{j=0}^{N}h^{j/2}U_{j,p}\left(\phi\right)\left(x,x'\right)+O\left(h^{\left(N+1-n\right)/2}\right).\label{eq:off diagonal expansion for kernel D}
\end{align}

Next to describe the on diagonal expansion for the Schwartz kernel
of the composition $B\phi\left(\frac{D_{h}}{\sqrt{h}}\right)$, one
has to again note that the Schwartz kernel of the pseudodifferential
operator $B=\frac{1}{\left(2\pi h\right)^{n}}\int e^{i\left(x-y\right)\frac{\xi}{h}}b\left(x,\xi;h\right)d\xi$
is expressed in an $h-$independent frame $\mathtt{l}_{0}$. Thus
we may write the composition

\begin{align}
 & e^{-\varphi_{h}\left(x\sqrt{h}\right)}B\phi\left(\frac{D_{h}}{\sqrt{h}}\right)\left(x\sqrt{h},x'\sqrt{h}\right)e^{\varphi_{h}\left(x'\sqrt{h}\right)}\nonumber \\
 & =\frac{1}{\left(2\pi h\right)^{n}}\int d\xi dy\,e^{-\frac{i\left(x\sqrt{h}-y\right).\xi}{h}}e^{-\varphi_{h}\left(x\sqrt{h}\right)}b\left(x\sqrt{h},\xi;h\right)e^{\varphi_{h}\left(y\right)}\underbrace{e^{-\varphi_{h}\left(y\right)}\phi\left(\frac{D_{h}}{\sqrt{h}}\right)\left(y,x'\sqrt{h}\right)e^{\varphi_{h}\left(x'\sqrt{h}\right)}}_{\eqqcolon K_{D,\phi}^{h}\left(\frac{y}{\sqrt{h}},x'\right)}\\
 & =\frac{1}{\left(2\pi h\right)^{n}}\int d\xi dy'\,e^{-\frac{i\left(x-y'\right).\xi}{\sqrt{h}}}e^{-\varphi_{h}\left(x\sqrt{h}\right)}b\left(x\sqrt{h},\xi;h\right)e^{\varphi_{h}\left(y'\sqrt{h}\right)}K_{D,\phi}^{h}\left(y',x'\right)\nonumber \\
 & =\frac{1}{\left(2\pi h\right)^{n}}\int d\xi dy'\,e^{-\frac{i\left(x-y'\right).\left(\xi+a\left(0\right)\right)}{\sqrt{h}}}b\left(x\sqrt{h},\xi;h\right)\underbrace{e^{-\varphi_{h}\left(x\sqrt{h}\right)-\frac{ix.a\left(0\right)}{\sqrt{h}}}}_{\eqqcolon e\left(x;\sqrt{h}\right)^{-1}}\underbrace{e^{\varphi_{h}\left(y'\sqrt{h}\right)+\frac{iy'.a\left(0\right)}{\sqrt{h}}}}_{\eqqcolon e\left(y';\sqrt{h}\right)}K_{D,\phi}^{h}\left(y',x'\right)\nonumber \\
 & =\frac{1}{\left(2\pi\sqrt{h}\right)^{n}}\int d\xi'dy'\,e^{-i\left(x-y'\right).\xi'}b\left(x\sqrt{h},\xi'\sqrt{h}-a;h\right)e\left(x;\sqrt{h}\right)^{-1}e\left(y';\sqrt{h}\right)K_{D,\phi}^{h}\left(y',x'\right)\label{eq:last line in demo}
\end{align}
having used to the two changes of variables $y'=\frac{y}{\sqrt{h}}$
and $\xi'=\frac{\xi+a\left(0\right)}{\sqrt{h}}$.

Now the exponential term in the last line above has an asymptotic
expansion in powers of $h^{1/2}$
\begin{align}
e\left(y;\sqrt{h}\right) & \coloneqq e^{\varphi_{h}\left(y\sqrt{h}\right)+\frac{iy.a\left(0\right)}{\sqrt{h}}}\sim1+\sum_{j=1}^{\infty}h^{j/2}e_{j}\left(y\right)\quad\textrm{ with each }\label{eq: expanding exponential}\\
e_{j}\left(y\right) & =\sum_{\left|\alpha\right|\leq j}y^{\alpha}e_{j,\alpha}
\end{align}
being a polynomial of degree atmost $j$, on account of \prettyref{eq:relation two frames}.
Plugging the above \prettyref{eq: expanding exponential}, the classical
symbolic expansion for $b\left(x,\xi;h\right)$ and \prettyref{eq:off diagonal expansion for kernel D}
gives the expansion \prettyref{eq: microlocal trace expansion} as
well as the calculation of the leading term \prettyref{eq:leading coefficient}.

To elucidate the structure \prettyref{eq: trace distribution structure}
of the coefficients, note that the symbolic/Taylor expansion of the
total symbol for $b$ in \prettyref{eq:last line in demo} yields
\begin{align}
b\left(x\sqrt{h},\xi'\sqrt{h}-a\left(0\right);h\right) & \sim\sum_{j=0}^{\infty}h^{j/2}b_{j}\left(x,\xi'\right)\quad\textrm{ with each }\label{eq:amplitude expansion}\\
b_{j}\left(\xi\right) & =\sum_{\left|\alpha\right|+\left|\beta\right|\leq j}x^{\alpha}\left(\xi'\right)^{\beta}b_{j,\alpha,\beta}\nonumber 
\end{align}
being polynomial in $\xi'$ of degree at most $j$. Plugging the last
equation above into \prettyref{eq:last line in demo} then gives that
each coefficient in \prettyref{eq: microlocal trace expansion} is
a sum of the form 
\[
U_{B,j,p}\left(\phi\right)\left(x,x'\right)=\sum_{\left|\alpha\right|+\left|\beta\right|+j'\leq j}c_{\alpha,\beta}x^{\alpha}\partial_{x}^{\beta}\left[U_{j',p}\left(x,x'\right)\left(\phi\right)\right].
\]
From here it follows that the distributions $U_{B,j,p}$ have the
same type of structure as was shown for $U_{j,p}$ in \cite[Prop. 7.2]{Savale2017-Koszul},
cf. \cite[Eq. 7.33]{Savale2017-Koszul} and following ones therein.
\end{proof}
By integrating the pointwise traces of the the distributions in \prettyref{eq: microlocal trace expansion},
we may further define 
\begin{eqnarray}
u_{B,j} & = & \int_{X}u_{B,j,p}dx,\quad\textrm{ with }\nonumber \\
u_{B,j,p} & \coloneqq & \textrm{tr }U_{B,j,p}\in C^{\infty}\left(X;\mathcal{S}'\left(\mathbb{R}_{s}\right)\right)\label{eq: pointwise trace distribution}
\end{eqnarray}
for $j=0,1,\ldots$. As with $U_{B,j,p}$, the distributions $u_{B,j}$
also have the same structure \prettyref{eq: trace distribution structure}.
In particular we have 
\begin{equation}
\textrm{sing spt}\left(u_{B,j}\right)\subset\mathbb{R}\setminus\left(-\sqrt{2\mu_{1}},\sqrt{2\mu_{1}}\right)\label{eq:uj smooth near 0}
\end{equation}
as with \cite[Cor. 7.3]{Savale2017-Koszul}.

The above now gives a corresponding generalization of \cite[Thm 1.3]{Savale2017-Koszul}.
Choose $f\in C_{c}^{\infty}\left(-\sqrt{2\mu_{1}},\sqrt{2\mu_{1}}\right)$.
With $0<T'<T_{0}$, let $\theta\in C_{c}^{\infty}\left(\left(-T_{0},T_{0}\right);\left[0,1\right]\right)$
such that $\theta\left(x\right)=1$ on $\left(-T',T'\right)$. Let
\begin{eqnarray*}
\mathcal{F}^{-1}\theta\left(x\right) & \coloneqq & \check{\theta}\left(x\right)=\frac{1}{2\pi}\int e^{ix\xi}\theta\left(\xi\right)d\xi\\
\mathcal{F}_{h}^{-1}\theta\left(x\right) & \coloneqq & \frac{1}{h}\check{\theta}\left(\frac{x}{h}\right)=\frac{1}{2\pi h}\int e^{\frac{i}{h}x\xi}\theta\left(\xi\right)d\xi
\end{eqnarray*}
be the given classical and semi-classical inverse Fourier transforms
respectively.
\begin{thm}
\label{thm:finer trace} There exist smooth functions $u_{B,j}\in C^{\infty}\left(-\sqrt{2\mu_{1}},\sqrt{2\mu_{1}}\right)$
such that there is a trace expansion 
\begin{align}
\textrm{tr}\left[Bf\left(\frac{D_{h}}{\sqrt{h}}\right)\left(\mathcal{F}_{h}^{-1}\theta\right)\left(\lambda\sqrt{h}-D_{h}\right)\right] & =\nonumber \\
\textrm{tr}\left[Bf\left(\frac{D_{h}}{\sqrt{h}}\right)\frac{1}{h}\check{\theta}\left(\frac{\lambda\sqrt{h}-D_{h}}{h}\right)\right] & =h^{-m-1}\left(\sum_{j=0}^{N-1}f\left(\lambda\right)u_{B,j}\left(\lambda\right)h^{j/2}+O\left(h^{N/2}\right)\right)\label{eq:Main trace expansion}
\end{align}
for each $N\in\mathbb{N}$,$\lambda\in\mathbb{R}$. 
\end{thm}

\begin{proof}
We break up the trace using 
\[
\theta\left(x\right)=\theta_{\epsilon}\left(x\right)+\underbrace{\left[\theta\left(x\right)-\theta_{\epsilon}\left(x\right)\right]}_{\vartheta\left(x\right)}
\]
 where $\theta_{\epsilon}\left(x\right)\coloneqq\theta\left(\frac{x}{h^{\epsilon}}\right)$
, $\epsilon\in\left(\frac{1}{4},\frac{1}{2}\right)$. The second function
in the break up satisfies $\vartheta\in C_{c}^{\infty}\left(\left(T'h^{\epsilon},T\right);\left[-1,1\right]\right)$
and one has 
\begin{equation}
\textrm{tr}\left[Bf\left(\frac{D_{h}}{\sqrt{h}}\right)\left(\mathcal{F}_{h}^{-1}\vartheta\right)\left(\lambda\sqrt{h}-D_{h}\right)\right]=O\left(h^{\infty}\right).\label{eq:main estimate from Koszul paper}
\end{equation}
The proof of the above is the same as \cite[Lem 3.1]{Savale2017-Koszul}
which already uses a microlocal partition of the trace cf. \cite[Eq. 3.1 and 3.2]{Savale2017-Koszul}.

Next we come to the trace involving $\theta_{\epsilon}\left(x\right)$
and write 
\begin{align}
\textrm{tr}\left[Bf\left(\frac{D_{h}}{\sqrt{h}}\right)\left(\mathcal{F}_{h}^{-1}\theta_{\epsilon}\right)\left(\lambda\sqrt{h}-D_{h}\right)\right] & =\textrm{tr}\left[Bf\left(\frac{D_{h}}{\sqrt{h}}\right)\frac{1}{h^{1-\epsilon}}\check{\theta}\left(\frac{\lambda\sqrt{h}-D_{h}}{h^{1-\epsilon}}\right)\right]\nonumber \\
 & =\frac{h^{-\frac{1}{2}}}{2\pi}\int\,dt\textrm{ tr}\left[Bf\left(\frac{D_{h}}{\sqrt{h}}\right)e^{it\left(\lambda-\frac{D_{h}}{\sqrt{h}}\right)}\right]\theta\left(th^{\frac{1}{2}-\epsilon}\right).\label{eq:local trace in terms of wave kernel}
\end{align}
Next, the expansion \prettyref{thm: microlocal trace expansion},
with $\phi\left(x\right)=f\left(x\right)e^{it\left(\lambda-x\right)}$,
combined with the smoothness of $u_{j}$ on $\textrm{spt}\left(f\right)\subset\left(-\sqrt{2\mu_{1}},\sqrt{2\mu_{1}}\right)$
\prettyref{eq:uj smooth near 0} gives 
\begin{eqnarray}
\textrm{tr}\left[Bf\left(\frac{D_{h}}{\sqrt{h}}\right)e^{it\left(\lambda-\frac{D_{h}}{\sqrt{h}}\right)}\right] & = & e^{it\lambda}h^{-n/2}\left(\sum_{j=0}^{N}h^{j/2}\widehat{fu_{B,j}}\left(t\right)\right)\nonumber \\
 &  & +h^{\left(N+1-n\right)/2}\underbrace{O\left(\sum_{k=0}^{n+1}\left\Vert \left\langle \xi\right\rangle ^{2N}\hat{\phi}^{\left(k\right)}\left(\xi-t\right)\right\Vert _{L^{1}}\right)}_{=O\left(\left\langle t\right\rangle ^{2N}\right)}.\label{eq:trace expansion to substitute}
\end{eqnarray}
Finally, plugging \prettyref{eq:trace expansion to substitute} into
\prettyref{eq:local trace in terms of wave kernel} and using $\theta\left(th^{\frac{1}{2}-\epsilon}\right)=1+O\left(h^{\infty}\right)$
gives 
\begin{eqnarray}
 &  & \frac{h^{-\frac{1}{2}}}{2\pi}\int\,dt\textrm{tr}\left[Bf\left(\frac{D_{h}}{\sqrt{h}}\right)e^{it\left(\lambda-\frac{D_{h}}{\sqrt{h}}\right)}\right]\theta\left(th^{\frac{1}{2}-\epsilon}\right)\nonumber \\
 & = & h^{-m-1}\left(\sum_{j=0}^{N}h^{j/2}f\left(\lambda\right)u_{B,j}\left(\lambda\right)\right)+O\left(h^{2N\left(\epsilon-\frac{1}{4}\right)-m-1}\right)\label{eq: final line finer exp}
\end{eqnarray}
via Fourier inversion as required.
\end{proof}

\subsection{Estimates in $\Psi_{\delta}^{0}$}

In the next section we shall also need estimates on the microlocal
trace \prettyref{eq: microlocal trace expansion} in the more general
class from \prettyref{subsec:The-class-,}. These follow from arguments
similar to the ones in the proofs of \prettyref{thm: microlocal trace expansion}
and \prettyref{thm:finer trace}. Firstly, note that the formula \prettyref{eq:leading coefficient}
shows that the microlocal Weyl measure of $\frac{D_{h}}{\sqrt{h}}$
concentrates on $\Sigma$ (cf. \cite[Thm 4.1]{Colin-de-Verdiere-Hillairet-TrelatI},
\cite[Thm. 25]{Savale-QC}). 

More generally, denote by $\pi_{\Sigma}:T^{*}X\rightarrow\Sigma,\;\pi_{\Sigma}\left(x,\xi\right)=\left(x,-a\left(x\right)\right)$
the projection onto the characteristic variety. Then for $B\in\Psi_{\delta}^{0}\left(X\right)$
the equations \prettyref{eq: rescaling Schw kernel}, \prettyref{eq: Diagonal kernel expansion}
and \prettyref{eq:last line in demo} imply 
\begin{align}
\textrm{tr }\left[Bf\left(\frac{D_{h}}{\sqrt{h}}\right)\right] & \leq C_{1}h^{-n/2}\mu\left(\pi_{\Sigma}\left(WF\left(B\right)\right)\right)+C_{2}h^{-n/2+\left(1/2-\delta\right)}\label{eq: Sdelta microlocal trace estimate}
\end{align}
with $\mu$ denoting the pullback Riemannian measure on $\Sigma$.
Above $C_{1}$, $C_{2}$ depend on appropriate $S_{\delta}^{0}$ semi-norms
\prettyref{eq: delta pseudodifferential estimates} of the Weyl symbol
of $B$.

For the estimate generalizing \prettyref{thm:finer trace}, the equations
corresponding to \prettyref{eq:main estimate from Koszul paper},
\prettyref{eq:local trace in terms of wave kernel} and \prettyref{eq: final line finer exp}
give
\begin{align}
\textrm{tr}\left[Bf\left(\frac{D_{h}}{\sqrt{h}}\right)\left(\mathcal{F}_{h}^{-1}\theta\right)\left(\lambda\sqrt{h}-D_{h}\right)\right] & \leq C_{1}h^{-m-1}\mu\left(\pi_{\Sigma}\left(WF\left(B\right)\right)\right)+C_{2}h^{-m-1+\left(1/2-\delta\right)},\label{eq:Sdelta microlocal trace estimate II}
\end{align}
with again $C_{1}$, $C_{2}$ depending on appropriate $S_{\delta}^{0}$
semi-norms \prettyref{eq: delta pseudodifferential estimates} of
the Weyl symbol of $B$.

\section{\label{sec:Eta-remainder-asymptotics}Eta remainder asymptotics}

In this section we shall prove the main \prettyref{thm: eta semiclassical limit}.

\subsection{\label{subsec:Partitions-adapted-to}Partitions adapted to recurrence}

We shall first choose a microlocal partition of unity adapted to the
recurrence sets $S_{T,\varepsilon}$ and $S_{T,\varepsilon}^{e}$
\prettyref{eq:recurrence set}. We recall that $\varepsilon=h^{\delta}$
for $\delta\in\left[0,\frac{1}{2}\right)$. With $\chi\in C_{c}^{\infty}\left(-1,1\right)$,
satisfying $\chi=1$ on $\left[-\frac{1}{2},\frac{1}{2}\right]$,
we choose a pseudodifferential operator $B=b^{W}\in\Psi_{\delta}^{0}\left(X\right)$
of the form
\begin{align*}
b & =\chi\left(\frac{\left|\xi+a\right|}{h^{\delta}}\right)b_{0}\left(x\right)\quad\textrm{ with }\\
b_{0} & =\begin{cases}
1 & \textrm{on }S_{T,\varepsilon}\\
0 & \textrm{on }\left(S_{T,\varepsilon}^{e}\right)^{c}
\end{cases}.
\end{align*}
The existence of $b_{0}$ satisfying the correct symbolic estimates
follows by an application of the Whitney extension theorem \cite[Sec. 2.3]{HormanderI}.
In particular this gives
\begin{equation}
\mu\left(\pi_{\Sigma}\left(WF\left(B\right)\right)\right)\leq\mu^{g^{TX}}\left(S_{T,\varepsilon}^{e}\right).\label{eq: choice of B near recurrence set}
\end{equation}

Next following \cite[Lem. 3.3]{Savale-Gutzwiller}, we note that near
each $x\in X\setminus S_{T,\varepsilon}^{e}$ there is a local Darboux
chart $\varphi_{x}:N_{x}\xrightarrow{\sim}C_{\varepsilon_{0}h^{\delta},T}\subset\mathbb{R}^{n}$
into a cylinder $C_{\varepsilon_{0}h^{\delta},T}\coloneqq B_{\mathbb{R}^{2m}}\left(\varepsilon_{0}h^{\delta}\right)\times\left(0,T\right)_{x_{0}}\subset\mathbb{R}_{x}^{n}$
of radius $\varepsilon_{0}h^{\delta}$ and length $T$ in Euclidean
space. For each such Darboux chart $\varphi_{x}:N_{x}\xrightarrow{\sim}C_{\varepsilon h^{\delta},T}\subset\mathbb{R}^{n}$
we set $N_{x}^{0}\coloneqq\varphi_{x}^{-1}\left(C_{\frac{\varepsilon h^{\delta}}{8},\frac{T}{8}}\right)$.
By compactness, we may choose a finite set of these such that $\bigcup_{u=1}^{N}N_{x_{u}}^{0}$,
$N=N_{h}=O\left(h^{-\delta}\right)$, cover $X\setminus S_{T,\varepsilon}^{e}$.
Denote by $\tilde{S}\subset T^{*}X$ the inverse image of any subset
$S\subset X$ under the projection $\pi:T^{*}X\rightarrow X$ and
by $c_{\delta}\coloneqq ch^{\delta}$ the $h$-dependent constant
for each $h$-independent constant $c$.

For $\delta\in\left[0,\frac{1}{2}\right),$ $\tau>0$, a $\left(\Omega,\tau,\delta\right)$-microlocal
partition of unity is defined to be a collection of zeroth-order self
-adjoint pseudo-differential operators 
\[
\mathcal{P}=\left\{ A_{u}\in\Psi_{\delta}^{0}\left(X\right)|0\leq u\leq N_{h}\right\} \cup\left\{ B\in\Psi_{\delta}^{0}\left(X\right)\right\} 
\]
satisfying
\begin{eqnarray}
\sum_{u=0}^{N_{h}}A_{u}+B & = & 1\nonumber \\
N_{h} & = & O\left(h^{-\delta}\right)\nonumber \\
WF\left(A_{0}\right)\subset & U_{0}\subset & \overline{T^{*}X}\setminus\Sigma_{\left[-\frac{\tau_{\delta}}{64},\frac{\tau_{\delta}}{64}\right]}^{D_{h}}\nonumber \\
WF\left(A_{u}\right)\Subset & U_{u}\subset & \Sigma_{\left[-\tau_{\delta},\tau_{\delta}\right]}^{D_{h}}\cap\tilde{N}_{x_{u}}^{0},\;1\leq u\leq N\nonumber \\
WF\left(B\right)\Subset & V\subset & \Sigma_{\left[-\tau_{\delta},\tau_{\delta}\right]}^{D_{h}}\cap\tilde{S}_{T,\varepsilon}^{e},\label{eq: microlocal partition of unity-1-1-1}
\end{eqnarray}
for some open cover $\left\{ U_{u}\right\} _{u=0}^{N}\cup V$ of $T^{*}X$
. For such a partition $\mathcal{P}$ define the pairs of indices
\begin{align}
I_{\mathcal{P}} & =\left\{ \left(u,u'\right)|u\leq u',\:WF\left(A_{u}\right)\cap WF\left(A_{u'}\right)\neq\emptyset\right\} \nonumber \\
J_{\mathcal{P}} & =\left\{ u|WF\left(A_{u}\right)\cap WF\left(B\right)\neq\emptyset\right\} .\label{eq: index sets}
\end{align}
 An augmentation $\left(\mathcal{P};\mathcal{V},\mathcal{W}\right)$
of this partition consists of an additional collection of open sets
\begin{align*}
\mathcal{V} & =\left\{ V_{uu'}^{1}\right\} _{\left(u,u'\right)\in I_{\mathcal{P}}}\cup\left\{ V_{u}^{2}\right\} _{u\in J_{\mathcal{P}}}\\
\mathcal{W} & =\left\{ W_{uu'}^{1}\right\} _{\left(u,u'\right)\in I_{\mathcal{P}}}\cup\left\{ W_{u}^{2}\right\} _{u\in J_{\mathcal{P}}}
\end{align*}
satisfying 
\begin{eqnarray}
WF\left(A_{u}\right)\cap WF\left(A_{u'}\right) & \subset & W_{uu'}^{1}\nonumber \\
 &  & \cap\nonumber \\
WF\left(A_{u}\right)\cup WF\left(A_{u'}\right) & \subset & V_{uu'}^{1}\Subset\Sigma_{\left[-2\tau_{\delta},2\tau_{\delta}\right]}^{D_{h}}\cap\tilde{N}_{x_{u}},\nonumber \\
WF\left(A_{u}\right)\cap WF\left(B\right) & \subset & W_{u}^{2}\nonumber \\
 &  & \cap\nonumber \\
WF\left(A_{u}\right)\cup WF\left(B\right) & \subset & V_{u}^{2}\Subset\Sigma_{\left[-2\tau_{\delta},2\tau_{\delta}\right]}^{D_{h}}\cap\tilde{N}_{x_{u}}.\label{eq:augmentation}
\end{eqnarray}
Next with $d=\sigma\left(D_{h}\right)$, for each pair of indices
in \prettyref{eq: index sets} we set 
\begin{align}
T_{uu'} & \coloneqq\frac{1}{\inf_{\left(g,\mathtt{v}\right)\in\mathcal{G}_{uu'}\times S_{\delta}^{0}\left(X;U\left(S\right)\right)}\left|H_{g,\mathtt{v}}d\right|},\label{eq: exit time 1}\\
S_{u} & \coloneqq\frac{1}{\inf_{\left(g,\mathtt{v}\right)\in\mathcal{H}_{u}\times S_{\delta}^{0}\left(X;U\left(S\right)\right)}\left|H_{g,\mathtt{v}}d\right|},\quad\textrm{ with}\label{eq: exit time 2}\\
\mathcal{G}_{uu'} & \coloneqq\left\{ g\in S_{\delta'}^{0}\left(T^{*}X;\left[0,1\right]\right)|\left.g\right|_{W_{uu'}^{1}}=1,\;\left.g\right|_{\left(V_{uu'}^{1}\right)^{c}}=0\right\} \label{eq: set of exit functions 1}\\
\mathcal{H}_{u} & \coloneqq\left\{ g\in S_{\delta'}^{0}\left(T^{*}X;\left[0,1\right]\right)|\left.g\right|_{W_{u}^{2}}=1,\;\left.g\right|_{\left(V_{u}^{2}\right)^{c}}=0\right\} \label{eq: set of exit functions 2}
\end{align}
and $\left|H_{g,\mathtt{v}}d\right|\coloneqq\sup\left\Vert \left\{ \mathtt{v}^{*}d\mathtt{v},g\right\} \right\Vert $
with the bracket being computed in terms of the chosen and induced
trivialization/coordinates on $N_{x_{u}},\tilde{N}_{x_{u}}$. A function
in $\mathcal{G}_{uu'}$ or $\mathcal{H}_{u}$ shall be referred to
as a trapping/microlocal weight function. Finally, the \textit{extension/trapping
time} of an augmented $\left(\Omega,\tau,\delta\right)$-partition
$\left(\mathcal{P};\mathcal{V},\mathcal{W}\right)$ is set to be 
\begin{equation}
T_{\left(\mathcal{P};\mathcal{V},\mathcal{W}\right)}\coloneqq\min\left\{ \min\left\{ T_{uu'}\right\} _{\left(u,u'\right)\in I_{\mathcal{P}}},\min\left\{ S_{u}\right\} _{u\in J_{\mathcal{P}}}\right\} .\label{eq: total extension time}
\end{equation}
 For each $\theta\in C_{c}^{\infty}\left(\mathbb{R}\right)$, $f\in\mathcal{S}\left(\mathbb{R}\right)$
and $A,B\in\Psi_{\delta}^{0}\left(X\right)$, we set 
\[
\mathcal{T}_{A,B}^{\theta}\left(D_{h}\right)\coloneqq\frac{1}{\pi}\int_{\mathbb{C}}\bar{\partial}\tilde{f}\left(z\right)\check{\theta}\left(\frac{\lambda-z}{\sqrt{h}}\right)\textrm{tr }\left[A\left(\frac{1}{\sqrt{h}}D_{h}-z\right)^{-1}B\right]dzd\bar{z}
\]
with $\tilde{f}$ being an almost analytic continuation of $f$. 

We now have the following. 
\begin{lem}
\label{lem: Lemma req Ehrenfest time} For each $\varepsilon=ch^{\delta}$,
$\delta\in\left[0,\frac{1}{2}\right)$, $c,\ell>0$, $\tau$ sufficiently
small and $T\leq\left(\frac{1}{2}-\delta\right)T_{E}^{\ell}\left(h\right)$,
there exists an augmented $\left(\Omega,\tau,\delta\right)$-partition
of unity $\left(\mathcal{P};\mathcal{V},\mathcal{W}\right)$ such
that 
\begin{equation}
T_{\left(\mathcal{P};\mathcal{V},\mathcal{W}\right)}>T.\label{eq: extension time arb large-1-1}
\end{equation}
Given $\theta\in C_{c}^{\infty}\left(\left(T_{0},T\right);\left[-1,1\right]\right)$
one further has

\begin{equation}
\mathcal{T}_{A_{u},A_{v}}^{\theta}\left(D_{h}\right),\:\mathcal{T}_{A_{u},B}^{\theta}\left(D_{h}\right),\textrm{ }\mathcal{T}_{B,A_{u}}^{\theta}\left(D_{h}\right)=O\left(h^{\infty}\right).\label{eq:Egorov type conclusion}
\end{equation}
\end{lem}

\begin{proof}
The proof of the first equation \prettyref{eq: extension time arb large-1-1}
is similar to \cite[Prop. 3.4]{Savale-Gutzwiller} with one caveat.
The construction of the microlocal weight functions $g_{u}$ \prettyref{eq: set of exit functions 1},
\prettyref{eq: set of exit functions 2} therein on \cite[pg. 1442]{Savale-Gutzwiller}
is done with respect to long Darboux charts of length $T$. These
charts, for points outside the recurrence set, are constructed using
the Reeb flow upto time $T$. The microlocal weight function $g_{u}$
further needs to lie in the symbol space $S_{\delta'}^{0}$, for some
$\delta'\in\left[0,\frac{1}{2}\right)$. Hence in passing from $S_{\delta'}^{0}$
symbolic estimates for the microlocal weight functions $g_{u}$, from
the flow dependent Darboux chart to $S_{\delta'}^{0}$ symbolic estimates
in a flow independent atlas one needs to account for the Jacobian
$de^{tR}$, $t\leq T$,  of the flow. In particular one needs the
estimates $\left|\partial_{x}^{\alpha}e^{tR}\right|\leq Ch^{-\left|\alpha\right|\left(\delta'-\delta\right)}$,
$t\leq T$, for some $\delta'\in\left[0,\frac{1}{2}\right)$. These
are seen to be valid for $\delta'=\frac{1}{2}-\left(\frac{1}{2}-\delta\right)\left(\frac{\ell}{\Lambda_{\textrm{max}}+\ell}\right)$,
$\ell>0$, and $T\leq\left(\frac{1}{2}-\delta\right)T_{E}^{\ell}\left(h\right)$
below a multiple of the Ehrenfest time.

Following the above, the proof of the second part \ref{eq:Egorov type conclusion}
is similar to that of \cite[Lemma 3.6]{Savale-Gutzwiller} or \cite[Lemma 3.1]{Savale2017-Koszul}.
\end{proof}

Next, we turn to examining the trace $\mathcal{T}_{B,B}^{\theta}\left(D_{h}\right)$.
We shall choose $\theta^{1}\in C_{c}^{\infty}\left(\left(-T_{0},T_{0}\right)\right)$
such that $\theta^{1}\left(x\right)=1$ on $\left(-T',T'\right)$,
$T'<T_{0}$, and $\check{\theta^{1}}\left(\xi\right)\geq0$. Further
we let $f\left(x\right)\geq0$ with $f\left(0\right)=1$. Since the
trace and trace norm of a positive self-adjoint operator agree, we
have
\begin{align}
 & \left\Vert Bf\left(\frac{D_{h}}{\sqrt{h}}\right)\left(\mathcal{F}_{h}^{-1}\theta^{1}\right)\left(\lambda\sqrt{h}-D_{h}\right)B\right\Vert _{\textrm{tr}}\nonumber \\
= & \textrm{tr }\left[B^{2}f\left(\frac{D_{h}}{\sqrt{h}}\right)\left(\mathcal{F}_{h}^{-1}\theta^{1}\right)\left(\lambda\sqrt{h}-D_{h}\right)\right]\nonumber \\
\leq & C_{1}h^{-m-1}\mu^{g^{TX}}\left(S_{T,\varepsilon}^{e}\right)+C_{2}h^{-m-1+\left(1/2-\delta\right)}\label{eq:trace vs trace norm}
\end{align}
for $h$-independent constants $C_{1}$, $C_{2}$ using \prettyref{eq:Sdelta microlocal trace estimate II}
and \prettyref{eq: choice of B near recurrence set}. For $\theta_{c}^{1}\left(t\right)\coloneqq\theta^{1}\left(t-c\right)$
one has $\mathcal{F}_{h}^{-1}\theta_{c}^{1}\left(x\right)=e^{i\frac{xc}{h}}\mathcal{F}_{h}^{-1}\theta^{1}\left(x\right)$.
Hence, $e^{ic\left(\lambda\sqrt{h}-D_{h}\right)}$ being a unitary
operator, the left hand side of \prettyref{eq:trace vs trace norm}
is unchanged under translation of $\theta^{1}$. By writing an arbitrary
$\theta\in C_{c}^{\infty}\left(T_{0},T\right)$, of possibly $h$-dependent
compact support $T=T\left(h\right)$, as a sum of translates of functions
with compact support in $\left(-T_{0},T_{0}\right)$ we obtain 
\begin{align*}
 & \left\Vert Bf\left(\frac{D_{h}}{\sqrt{h}}\right)\left(\mathcal{F}_{h}^{-1}\theta\right)\left(\lambda\sqrt{h}-D_{h}\right)B\right\Vert _{\textrm{tr}}\\
\leq & T\left[C_{1}h^{-m-1}\mu^{g^{TX}}\left(S_{T,\varepsilon}^{e}\right)+C_{2}h^{-m-1+\left(1/2-\delta\right)}\right].
\end{align*}
Combined with \prettyref{eq:Egorov type conclusion} and \cite[Thm 1.3]{Savale2017-Koszul}
we have 
\begin{align}
 & \left|\textrm{tr }\left[f\left(\frac{D_{h}}{\sqrt{h}}\right)\left(\mathcal{F}_{h}^{-1}\theta\right)\left(\lambda\sqrt{h}-D_{h}\right)\right]-h^{-m-1}f\left(\lambda\right)u_{0}\left(\lambda\right)\right|\nonumber \\
\leq & T\left[C_{1}h^{-m-1}\mu^{g^{TX}}\left(S_{T,\varepsilon}^{e}\right)+C_{2}h^{-m-1+\left(1/2-\delta\right)}\right]\label{eq:leading part trace exp}
\end{align}
 for arbitrary $\theta\in C_{c}^{\infty}\left(-T,T\right)$, of possibly
$h$-dependent compact support $T=T\left(h\right)<\left(\frac{1}{2}-\delta\right)T_{E}^{\ell}\left(h\right)$.

\subsection{\label{subsec:Tauberian-argument}Tauberian argument}

 Next we derive asymptotics for irregular functional traces of $\frac{D_{h}}{\sqrt{h}}$
via a Tauberian argument. The arguments below are modified from semiclassical
Tauberian arguments as in \cite[Ch. 11]{Dimassi-Sjostrand}. They
should be also compared with classical Tauberian arguments as in \cite[Appx. B]{Safarov-Vassiliev-1997}.

First note from \cite[Cor. 7.3]{Savale2017-Koszul} that the distributions
$u_{j}\in\mathcal{S}'\left(\mathbb{R}\right)$ \prettyref{eq: microlocal trace expansion}
are smooth near $0$. Hence 
\begin{equation}
u_{j}^{\pm}\left(x\right)\coloneqq1_{\left[0,\infty\right)}\left(\pm x\right)u_{j}\left(x\right)\in\mathcal{S}'\left(\mathbb{R}\right)\label{eq:cutoff functions}
\end{equation}
are well defined tempered distributions and we similarly define $f^{\pm}$
for any $f\in\mathcal{S}\left(\mathbb{R}\right)$. We then have the
following.
\begin{lem}
For any $f\in\mathcal{S}\left(\mathbb{R}\right)$, $\varepsilon=ch^{\delta}$,
$\delta\in\left[0,\frac{1}{2}\right)$, $c,\ell>0$ and $T\leq\left(\frac{1}{2}-\delta\right)T_{E}^{\ell}\left(h\right)$,
one has
\begin{align}
 & \left|\textrm{tr }f^{\pm}\left(\frac{D_{h}}{\sqrt{h}}\right)-\left[h^{-m-\frac{1}{2}}u_{0}^{\pm}\left(f\right)+h^{-m}u_{1}^{\pm}\left(f\right)\right]\right|\nonumber \\
\leq & h^{-m}\left\Vert f\right\Vert _{C^{0}}\left[u_{0}\left(0\right)T^{-1}+O\left(T^{-2}+\mu\left(S_{T,\varepsilon}^{e}\right)\right)\right].\label{eq:irreg funct tr exp}
\end{align}
\end{lem}

\begin{proof}
First choose $\theta\in\mathcal{S}\left(\mathbb{R}\right)$ such
that $\check{\theta}\geq\frac{1}{1+\epsilon}$, $\epsilon>0$, on
$\left[0,1\right]$ and $1=\theta\left(0\right)=\int d\xi\check{\theta}\left(\xi\right)$.
Set $\theta_{T}\left(x\right)=\theta\left(T^{-1}x\right)$ and let
$N\left(a,b\right)$ denote the number of eigenvalues of $\frac{D_{h}}{\sqrt{h}}$
in the interval $\left(a,b\right)$. Choosing $f\left(x\right)\geq0$,
the trace expansion \prettyref{eq:leading part trace exp} with $\lambda=0$
now gives 
\begin{align*}
\frac{T}{\left(1+\epsilon\right)h}N\left(0,T^{-1}\sqrt{h}\right)\leq & \textrm{tr}\left[f\left(\frac{D_{h}}{\sqrt{h}}\right)\frac{T}{h}\check{\theta}\left(\frac{-TD_{h}}{h}\right)\right]\\
= & h^{-m-1}\left[f\left(0\right)u_{0}\left(0\right)+O\left(T\left[\mu\left(S_{T,\varepsilon}^{e}\right)+h^{1/2-\delta}\right]\right)\right]
\end{align*}
$\forall\epsilon>0$. And hence 
\begin{equation}
N\left(0,T^{-1}\sqrt{h}\right)\leq h^{-m}\left[T^{-1}f\left(0\right)u_{0}\left(0\right)+O\left(\mu\left(S_{T,\varepsilon}^{e}\right)+h^{1/2-\delta}\right)\right].\label{eq:local Weyl law}
\end{equation}

By virtue of \prettyref{eq: local on diagonal expansion}, we may
assume $f\in C_{c}^{\infty}\left(-\sqrt{2\mu_{1}},\sqrt{2\mu_{1}}\right)$.
The spectral measure for $\frac{D_{h}}{\sqrt{h}}$ is defined as $\mathfrak{M}_{f}\left(\lambda'\right)\coloneqq\sum_{\lambda\in\textrm{Spec}\left(\frac{D_{h}}{\sqrt{h}}\right)}f\left(\lambda\right)\delta\left(\lambda-\lambda'\right)$.
Next we choose an even function $\theta\in\mathcal{S}\left(\mathbb{R}\right)$
such that its transform satisfies $\textrm{spt}\left(\check{\theta}\right)\subset\left[-1,1\right]$,
$1\geq\check{\theta}\left(\xi\right)\geq0$ , $\int\check{\theta}\left(\xi\right)d\xi=1$.
Setting $\theta_{\frac{1}{2}}\left(x\right)=\theta\left(\frac{x}{\sqrt{h}}\right)$,
\cite[Thm 1.3]{Savale2017-Koszul} up to its first two terms can
be written as 
\[
\mathfrak{M}_{f}\ast\left(\mathcal{F}_{h}^{-1}\theta_{\frac{1}{2}}\right)\left(\lambda\right)=h^{-m-\frac{1}{2}}\left(f\left(\lambda\right)u_{0}\left(\lambda\right)+h^{1/2}f\left(\lambda\right)u_{1}\left(\lambda\right)+O\left(h\right)\right).
\]
Both sides above involving Schwartz functions in $\lambda$, the remainder
can be replaced by $O\left(\frac{h}{\left\langle \lambda\right\rangle ^{2}}\right)$.
Integrating further gives
\begin{align}
 & \int_{-\infty}^{0}d\lambda\int d\lambda'\left(\mathcal{F}_{h}^{-1}\theta_{\frac{1}{2}}\right)\left(\lambda-\lambda'\right)\mathfrak{M}_{f}\left(\lambda'\right)\label{eq:main trace exp integrated}\\
= & h^{-m-\frac{1}{2}}\left(\int_{-\infty}^{0}d\lambda f\left(\lambda\right)u_{0}\left(\lambda\right)+h^{1/2}\int_{-\infty}^{0}d\lambda f\left(\lambda\right)u_{1}\left(\lambda\right)+O\left(h\right)\right).\nonumber 
\end{align}
Now note that
\begin{equation}
\int_{-\infty}^{0}d\lambda\left(\mathcal{F}_{h}^{-1}\theta_{\frac{1}{2}}\right)\left(\lambda-\lambda'\right)=1_{\left(-\infty,0\right]}\left(\lambda'\right)+\phi\left(\frac{\lambda'}{\sqrt{h}}\right)\label{eq: definition remainder}
\end{equation}
where $\phi\left(x\right)\coloneqq\int_{-\infty}^{0}dt\check{\theta}\left(t-x\right)-1_{\left(-\infty,0\right]}\left(x\right)$
is a function that is rapidly decaying with all derivatives, odd and
smooth on $\mathbb{R}_{x}\setminus0$ and satisfies $\phi'\left(x\right)=\check{\theta}\left(-x\right)$
for $x\neq0$. 

Next with $x\geq0$ we compute 
\begin{align}
 & \left|\phi\left(x\right)-\phi\ast\check{\theta}_{T}\left(x\right)\right|\nonumber \\
= & \left|\int dy\left[\phi\left(x\right)-\phi\left(x-T^{-1}y\right)\right]\check{\theta}\left(y\right)\right|\nonumber \\
\leq & \int_{y\leq xT}dy\left|\phi'\left(c\left(x,y\right)\right)\right|T^{-1}\left|y\right|\check{\theta}\left(y\right)+2\int_{y\geq xT}dy\check{\theta}\left(y\right)\nonumber \\
\leq & T^{-1}\underbrace{\int_{-\infty}^{xT}dy\left|y\right|\check{\theta}\left(y\right)}_{=\theta_{1}\left(xT\right)}+2\underbrace{\int_{y\geq xT}dy\check{\theta}\left(y\right)}_{=\theta_{2}\left(xT\right)}\label{eq:replace with convolution}
\end{align}
where $c\left(x,y\right)\in\left[x-T^{-1}y,x\right]$. A similar estimate
holds for $x\leq0$.

Next pairing the second term of \prettyref{eq:replace with convolution}
with $\mathfrak{M}_{f}\left(\lambda'\right)$ gives 
\begin{align}
\int d\lambda'\theta_{2}\left(\frac{\lambda'T}{\sqrt{h}}\right)\mathfrak{M}_{f}\left(\lambda'\right)\leq & h^{-m}\left[T^{-1}\left\Vert f\right\Vert _{C^{0}}u_{0}\left(0\right)+O\left(\mu\left(S_{T,\varepsilon}^{e}\right)+h^{1/2-\delta}\right)\right]\label{eq:pairing with first term}
\end{align}
on covering $\mathbb{R}_{\lambda'}$ with intervals of size $O\left(T^{-1}h\right)$
and using the Weyl estimate \prettyref{eq:local Weyl law}. A similar
estimate 
\begin{equation}
\int d\lambda'T^{-1}\theta_{1}\left(\frac{\lambda'T}{\sqrt{h}}\right)\mathfrak{M}_{f}\left(\lambda'\right)=O\left(h^{-m}T^{-1}\left[T^{-1}+\mu\left(S_{T,\varepsilon}^{e}\right)\right]\right)\label{eq: second term}
\end{equation}
then gives 
\begin{align}
 & \int d\lambda'\left[\phi_{R}\left(\frac{\lambda'}{\sqrt{h}}\right)-\phi_{R}\ast\check{\theta}_{T}\left(\frac{\lambda'}{\sqrt{h}}\right)\right]\mathfrak{M}_{f}\left(\lambda'\right)\label{eq:cutoff vs cutoff conv}\\
\leq & h^{-m}\left[\left\Vert f\right\Vert _{C^{0}}u_{0}\left(0\right)T^{-1}+O\left(T^{-2}+\mu\left(S_{T,\varepsilon}^{e}\right)\right)\right].\nonumber 
\end{align}
on combining\prettyref{eq:replace with convolution}, \prettyref{eq:pairing with first term}
and \prettyref{eq: second term}. 

The second term above is estimated on integrating \prettyref{eq:leading part trace exp}
against $\phi$ as
\begin{align}
\int d\lambda'\phi\ast\check{\theta}_{T}\left(\frac{\lambda'}{\sqrt{h}}\right)\mathfrak{M}_{f}\left(\lambda'\right) & =h^{-m}\left[\intop d\lambda\phi\left(\lambda\right)f\left(0\right)u_{0}\left(0\right)+O\left(\mu\left(S_{T,\varepsilon}^{e}\right)+h^{1/2-\delta}\right)\right]\nonumber \\
 & =O\left(h^{-m}\left[\mu\left(S_{T,\varepsilon}^{e}\right)+h^{1/2-\delta}\right]\right)\label{eq: smoothed function expansion}
\end{align}
since $\phi$ is an odd function. Finally combining \prettyref{eq:main trace exp integrated},
\prettyref{eq: definition remainder}, \prettyref{eq:cutoff vs cutoff conv}
and \prettyref{eq: smoothed function expansion} gives 
\begin{align*}
\textrm{tr }f^{-}\left(\frac{D_{h}}{\sqrt{h}}\right) & =\int d\lambda'1_{\left(-\infty,0\right]}\left(\lambda'\right)\mathfrak{M}_{f}\left(\lambda'\right)\\
 & =h^{-m-\frac{1}{2}}\left(\int_{-\infty}^{0}d\lambda f\left(\lambda\right)u_{0}\left(\lambda\right)+h^{1/2}\int_{-\infty}^{0}d\lambda f\left(\lambda\right)u_{1}\left(\lambda\right)\right)\\
 & \qquad+\int d\lambda'\phi\left(\frac{\lambda'}{\sqrt{h}}\right)\mathfrak{M}_{f}\left(\lambda'\right)\\
 & =h^{-m-\frac{1}{2}}\left(\int_{-\infty}^{0}d\lambda f\left(\lambda\right)u_{0}\left(\lambda\right)+h^{1/2}\int_{-\infty}^{0}d\lambda f\left(\lambda\right)u_{1}\left(\lambda\right)\right)+R\left(h\right),\\
\textrm{with }\quad R\left(h\right) & \leq h^{-m}\left[\left\Vert f\right\Vert _{C^{0}}u_{0}\left(0\right)T^{-1}+O\left(T^{-2}+\mu\left(S_{T,\varepsilon}^{e}\right)\right)\right],
\end{align*}
as required. 
\end{proof}
We now prove our main result \prettyref{thm: eta semiclassical limit}.
\begin{proof}[Proof of \prettyref{thm: eta semiclassical limit}]
 The eta invariant being unchanged by positive scaling, we have 
\begin{align}
\eta\left(D_{h}\right)=\eta\left(\frac{D_{h}}{\sqrt{h}}\right) & =\int_{0}^{\infty}dt\frac{1}{\sqrt{\pi t}}\textrm{ tr}\left[\frac{D_{h}}{\sqrt{h}}e^{-\frac{t}{h}D_{h}^{2}}\right]\nonumber \\
 & =\int_{0}^{1}dt\frac{1}{\sqrt{\pi t}}\textrm{ tr}\left[\frac{D_{h}}{\sqrt{h}}e^{-\frac{t}{h}D_{h}^{2}}\right]+\int_{1}^{\infty}dt\frac{1}{\sqrt{\pi t}}\textrm{ tr}\left[\frac{D_{h}}{\sqrt{h}}e^{-\frac{t}{h}D_{h}^{2}}\right].\label{eq:integral breakup}
\end{align}
Using \cite[Prop. 3.4 and Eq. 4.5]{Savale-Asmptotics} the first integral
is seen to give 
\begin{equation}
\int_{0}^{1}dt\frac{1}{\sqrt{\pi t}}\textrm{ tr}\left[\frac{D_{h}}{\sqrt{h}}e^{-\frac{t}{h}D_{h}^{2}}\right]=h^{-m}\left[\int_{0}^{1}dt\frac{1}{\sqrt{\pi t}}u_{1}\left(se^{-ts^{2}}\right)\right]+O\left(h^{-m+1}\right).\label{eq:first integral}
\end{equation}

While the second integral is evaluated to be $\textrm{tr }E\left(\frac{D_{h}}{\sqrt{h}}\right)$
where 
\[
E\left(x\right)\coloneqq\text{sign}(x)\text{erfc}(|x|)=\text{sign}(x)\cdot\frac{2}{\sqrt{\pi}}\int_{|x|}^{\infty}e^{-s^{2}}ds
\]
 where we use the convention $\text{sign}(0)=0$. The above function
$E$ being odd and the difference of two functions of the form \prettyref{eq:cutoff functions}
we obtain 
\begin{equation}
\left|\textrm{tr }E\left(\frac{D_{h}}{\sqrt{h}}\right)-h^{-m-\frac{1}{2}}\left[u_{0}\left(E\right)\right]+h^{-m}\left[u_{1}\left(E\right)\right]\right|\leq h^{-m}\left[u_{0}\left(0\right)T^{-1}+O\left(\mu\left(S_{T,\varepsilon}^{e}\right)+T^{-2}\right)\right].\label{eq:second integral}
\end{equation}
From \cite[Prop. 7.4]{Savale2017-Koszul}, $u_{0}\left(\lambda\right)$
is an even function of $\lambda$. Hence the first evaluation on the
left hand side above is $0$. The second by definition is
\begin{equation}
u_{1}\left(E\right)=\int_{1}^{\infty}dt\frac{1}{\sqrt{\pi t}}u_{1}\left(se^{-ts^{2}}\right).\label{eq:second evaluation}
\end{equation}
Finally combining \prettyref{eq:integral breakup}, \prettyref{eq:first integral},
\prettyref{eq:second integral}, \prettyref{eq:second evaluation}
with the computation in \cite[Cor. 7.3]{Savale-Gutzwiller} gives
\begin{align*}
 & \left|\eta_{h}-h^{-m}\left(-\frac{1}{2}\frac{1}{\left(2\pi\right)^{m+1}}\frac{1}{m!}\int_{X}\left[\textrm{tr }\frac{1}{\left|\mathfrak{J}\right|}\left(\nabla^{TX}\mathfrak{J}\right)^{0}\right]a\wedge\left(da\right)^{m}\right)\right|\\
\leq & h^{-m}\left[u_{0}\left(0\right)T^{-1}+O\left(\mu\left(S_{T,\varepsilon}^{e}\right)+T^{-2}\right)\right]
\end{align*}
The theorem now follows from the computation $u_{0}\left(0\right)=\frac{\det\left|\mathfrak{J}\right|}{\left(4\pi\right)^{n/2}}$
done in \cite[Prop. 7.4]{Savale2017-Koszul}.
\end{proof}

\section{\label{sec:Examples-of-recurrence}Examples of recurrence}

In this section we prove the Corollaries of the main \prettyref{thm: eta semiclassical limit}.
These shall be based on asymptotic volume estimates for the recurrence
set in particular cases. 

\subsection{\label{subsec:Anosov-flows}Anosov flows}

The Reeb flow is Anosov if there exists a constant $c_{1}>0$ and
an invariant continuous splitting
\begin{align}
TX & =\mathbb{R}\left[R\right]\oplus E^{u}\oplus E^{s}\quad\textrm{ such that}\nonumber \\
\left\Vert \left.e^{tR}\right|_{E^{s}}\right\Vert  & \leq e^{-c_{1}t},\nonumber \\
\left\Vert \left.e^{-tR}\right|_{E^{u}}\right\Vert  & \leq e^{-c_{1}t},\label{eq:Anosov condition}
\end{align}
$\forall t>0,$ where the norm is taken with respect to some restricted
Riemannian metric $g$. Below it shall be further useful to choose
a metric $g$ for which the Reeb orbits are geodesics, any metric
satisfying $g\left(R,.\right)=\lambda$ has this property. The bounds
on the recurrence set in this case are exponential and proved in Proposition
\prettyref{prop:Anosov lifted recurrence bound} below. We shall first
need some related concepts about Anosov flows.

Each $T>0$ defines the Bowen distance on $X$ via 
\[
d_{T}^{g}\left(x_{1},x_{2}\right)\coloneqq\sup_{t\in\left[0,T\right]}d^{g}\left(e^{tR}x_{1},e^{tR}x_{2}\right).
\]
A $\left(T,\epsilon\right)$ separated subset $S\subset X$ is finite
set in which any two distinct points are at least distance $\epsilon$
apart with respect to the above $d_{T}$. Denote by $N\left(T,\epsilon\right)$
the maximum cardinality of a $\left(T,\epsilon\right)$ separated
set in $X$. The topological entropy of the flow is now defined 
\begin{equation}
\mathtt{h}_{\textrm{top}}=\mathtt{h}_{\textrm{top}}\left(e^{R}\right)\coloneqq\lim_{\epsilon\rightarrow0}\left(\limsup_{T\rightarrow\infty}\frac{\ln N\left(T,\epsilon\right)}{T}\right)\label{eq:topological entropy}
\end{equation}
and is known to be a topological invariant.

Next for each $s\in\left(0,1\right]$, let $\mathcal{D}_{s}\left(X\right)$,
$\mathcal{D}_{s-}\left(X\right)$ respectively be the set of compatible
distorted distance functions $d$ on the manifold satisfying 
\begin{align*}
d^{g} & \lesssim d\lesssim\left(d^{g}\right)^{s}\\
d^{g} & \lesssim d\lesssim\left(d^{g}\right)^{s-\epsilon},\quad\textrm{ for some }\epsilon>0,
\end{align*}
respectively. It is clear that 
\begin{align*}
\mathcal{D}_{s}\left(X\right) & \subset\mathcal{D}_{s'}\left(X\right),\\
\mathcal{D}_{s-}\left(X\right) & \subset\mathcal{D}_{s'-}\left(X\right),\quad s'<s,
\end{align*}
while all distances in $\mathcal{D}_{s}\left(X\right)$, $\mathcal{D}_{s-}\left(X\right)$
define the same manifold topology. Furthermore $\mathcal{D}_{1}\left(X\right)$
is the set of all distances equivalent to the $d^{g}$ and hence includes
all Riemannian distances. To each distance $d\in\mathcal{D}_{s}\left(X\right),\,\mathcal{D}_{s-}\left(X\right)$
is associated the Lipschitz constant of the time one flow
\begin{equation}
L_{d}=L_{d}\left(e^{R}\right)\coloneqq\sup_{x_{1}\neq x_{2}}\frac{d\left(e^{R}x_{1},e^{R}x_{2}\right)}{d\left(x_{1},x_{2}\right)}.\label{eq:Lipschitz constant}
\end{equation}
It shall also be useful to define the local skewness of the time one
map
\[
SL_{d}\left(e^{R}\right)\coloneqq\sup_{\varepsilon>0}\inf_{0<d\left(x_{1},x_{2}\right)<\varepsilon}\frac{d\left(e^{R}x_{1},e^{R}x_{2}\right)}{d\left(x_{1},x_{2}\right)}.
\]
One now has the following inequalities for topological entropy of
an Anosov flow.
\begin{lem}
\label{lem:Entropy inequality} The topological entropy \prettyref{eq:topological entropy}
satisfies the inequalities
\[
\frac{n}{2}\left(\inf_{d\in\mathcal{D}_{\frac{1}{2}-}\left(X\right)}\ln L_{d}\right)\leq\mathtt{h}_{\textrm{top}}\leq n\left(\inf_{d\in\mathcal{D}_{\frac{1}{2}-}\left(X\right)}\ln L_{d}\right)
\]
with respect to the infimum of the log Lipschitz constants \prettyref{eq:Lipschitz constant}
in $\mathcal{D}_{\frac{1}{2}-}\left(X\right)$.
\end{lem}

\begin{proof}
With $\textrm{HD}\left(d\right)$ denoting the Hausdorff dimension
of the manifold with respect to the distance $d$, the inequalities
\begin{equation}
\textrm{HD}\left(d\right)\ln SL_{d}\leq\mathtt{h}_{\textrm{top}}\leq\textrm{HD}\left(d\right)\ln L_{d}\label{eq:SL vs htop vs L}
\end{equation}
$\forall d\in\mathcal{D}_{\frac{1}{2}-}\left(X\right)$, are fairly
well known (\cite{Fathi89,Roth2020}, cf. \cite[Thm. 7.15]{Walters-book-82}).
Furthermore, from the definition $\frac{n}{2}\leq\textrm{HD}\left(d\right)\leq n$.
This hence proves the lemma in one direction
\begin{equation}
\mathtt{h}_{\textrm{top}}\leq n\left(\inf_{d\in\mathcal{D}_{\frac{1}{2}-}\left(X\right)}\ln L_{d}\right).\label{eq:comparison of flow invariants}
\end{equation}

It now remains to construct a sequence of distances $d_{k}\in\mathcal{D}_{\frac{1}{2}-}\left(X\right)$,
$k=1,2,\ldots$ with $n\,\ln L_{d_{k}}$ approaching $2\mathtt{h}_{\textrm{top}}$
as $k\rightarrow\infty$. A sequence of distances $d_{k}$ as desired
can be constructed for expansive maps \cite{Fathi89,Roth2020}, cf.
also the related construction of the Hamenstädt distance \cite{Hamenstadt89}.
The time one map $e^{R}$ is unfortunately not expansive in the flow
direction. The lack of expansiveness can however be replaced with
the following instability property which is satisfied by $e^{R}$
\cite{Norton-Obrien73}: there exists a positive constant $c_{2}>0$
such that the following implication holds
\begin{equation}
y\neq e^{tR}x,\,\forall t\in\mathbb{R}\implies d^{g}\left(e^{jR}x,e^{jR}y\right)>c_{2}\;\textrm{ for some }j\in\mathbb{Z}.\label{eq:instability of Anosov flows}
\end{equation}
In fact the proof therein gives the stronger statement: for any $\epsilon>0$
there exist positive constants $c>0$, $\alpha>1$ such that for $\alpha_{\epsilon}\coloneqq\alpha+\epsilon$
the stronger implication holds
\begin{align*}
 & y\neq e^{tR}x,\,\forall t\in\mathbb{R},\;d^{g^{TX}}\left(x,y\right)<c_{2},\qquad\\
\implies & \alpha d^{g^{TX}}\left(x,y\right)\leq\max\left\{ d^{g}\left(e^{R}x,e^{R}y\right),d^{g}\left(e^{-R}x,e^{-R}y\right)\right\} \leq\alpha_{\epsilon}d^{g^{TX}}\left(x,y\right).
\end{align*}
The constant $\alpha$ is related to the exponent $c_{1}$ in the
Anosov condition \prettyref{eq:Anosov condition}.

We then define 
\[
N\left(x,y\right)\coloneqq\begin{cases}
\infty, & x=y,\\
\inf\left\{ N\in\mathbb{N}_{0}|\max_{j\in\left[-N,N\right]}d^{g^{TX}}\left(e^{jR}x,e^{jR}y\right)>c_{2}\alpha^{-\left|j\right|}\right\} , & x\neq y.
\end{cases}
\]
The following bounds are easily derived 
\begin{equation}
\max\left\{ 0,\frac{\ln\frac{c_{2}}{d\left(x,y\right)}}{\ln\alpha\alpha_{\epsilon}}\right\} \leq N\left(x,y\right)\leq\max\left\{ 0,\frac{\ln\frac{c_{2}}{d\left(x,y\right)}}{\ln\alpha}\right\} .\label{eq:bounds on N}
\end{equation}
Next set 
\begin{align}
\rho\left(x,y\right)\coloneqq & \alpha^{-N\left(x,y\right)},\quad\textrm{ satisfying }\label{eq:choice of alpha}\\
\frac{d^{g}\left(x,y\right)}{c_{2}}\leq\rho\left(x,y\right) & \leq\left[\frac{d^{g}\left(x,y\right)}{c_{2}}\right]^{\ln\alpha/\ln\left(\alpha\alpha_{\epsilon}\right)}\quad\textrm{ for }d^{g}\left(x,y\right)\leq c_{2}.\label{eq:d vs rho}
\end{align}
Thus it follows that $\rho$ defines the same manifold topology as
$d^{g^{TX}}$, although it does not quite define a distance. From
\prettyref{eq:bounds on N} one further has $d^{g^{TX}}\left(x,y\right)\geq\frac{c_{2}}{2}\implies N\left(x,y\right)\leq\frac{\ln2}{\ln\alpha}\implies\alpha^{N}\leq\alpha^{\frac{\ln2}{\ln\alpha}}=2$.
An application of the triangle inequality for $d^{g^{TX}}$ gives
\begin{align*}
\min\left\{ N\left(x,y\right),N\left(y,z\right)\right\}  & \leq M+N\left(x,z\right)\quad\textrm{ and }\\
\quad\rho\left(x,z\right) & \leq2\max\left\{ \rho\left(x,y\right),\rho\left(y,z\right)\right\} \quad\forall x,y,z\in X
\end{align*}
as a weak triangle inequality for $\rho$. An application of Frink's
metrization theorem then gives the existence of a metric $D$ on $X$
satisfying 
\begin{equation}
D\left(x,y\right)\leq\rho\left(x,y\right)\leq4D\left(x,y\right)\label{eq:D vs rho comparison}
\end{equation}
and hence defining the same topology as $d^{g^{TX}}$. On account
of \prettyref{eq:d vs rho} and \prettyref{eq:D vs rho comparison}
we have $D\in\mathcal{D}_{\frac{1}{2}-}\left(X\right)$.

Next, it is an exercise to show that $\rho\left(e^{jR}x,e^{jR}y\right)\leq\alpha^{j}\rho\left(x,y\right)$
with equality on some neighborhood $V_{j}\subset X\times X$ of the
diagonal in the product. Using \prettyref{eq:D vs rho comparison}
this yields 
\begin{align}
D\left(e^{jR}x,e^{jR}y\right) & \leq4\alpha^{j}D\left(x,y\right),\quad\forall x,y\in X,\nonumber \\
D\left(e^{jR}x,e^{jR}y\right) & \geq\frac{1}{4}\alpha^{j}D\left(x,y\right),\quad\forall\left(x,y\right)\in V_{j}.\label{eq:Lip =000026 SL est for D}
\end{align}
And hence 
\begin{equation}
L_{D}\left(e^{jR}\right)\leq4\alpha^{j}\leq16SL_{D}\left(e^{jR}\right).\label{eq: L vs SL for D}
\end{equation}
Next we define the following distance 
\begin{equation}
d_{k}\left(x,y\right)\coloneqq\max_{0\leq j\leq k-1}\frac{D\left(e^{jR}x,e^{jR}y\right)}{L_{D}^{j/n}}\label{eq: approximating metrics}
\end{equation}
which is equivalent to $D$. Its Lipschitz constant $L_{d_{k}}\left(e^{R}\right)=\left[L_{D}\left(e^{kR}\right)\right]^{1/k}$
is seen to be given in terms of the $D$-Lipschitz constant of the
time $k$ map. Using \prettyref{eq:SL vs htop vs L}, \prettyref{eq: L vs SL for D}
and $\mathtt{h}_{\textrm{top}}\left(e^{kR}\right)=k\mathtt{h}_{\textrm{top}}\left(e^{R}\right)$
this finally gives 
\begin{align*}
n\left(\frac{\ln\alpha}{\ln\left(\alpha\alpha_{\epsilon}\right)}\right)\ln L_{d_{k}}\left(e^{R}\right) & \leq\textrm{HD}\left(d_{k}\right)\ln L_{d_{k}}\left(e^{R}\right)\\
 & \leq\frac{\textrm{HD}\left(d_{k}\right)}{k}\ln L_{D}\left(e^{kR}\right)\\
 & \leq\frac{\textrm{HD}\left(d_{k}\right)}{k}\left[\ln16+\ln SL_{D}\left(e^{kR}\right)\right]\\
 & \leq\frac{\textrm{HD}\left(d_{k}\right)}{k}\ln16+\frac{1}{k}\mathtt{h}_{\textrm{top}}\left(e^{kR}\right)\\
 & =\frac{\textrm{HD}\left(d_{k}\right)}{k}\ln16+\mathtt{h}_{\textrm{top}}\left(e^{R}\right)\\
 & \leq\frac{n}{k}\ln16+\mathtt{h}_{\textrm{top}}\left(e^{R}\right).
\end{align*}
Letting $k\rightarrow\infty$ and observing $\left(\frac{\ln\alpha}{\ln\left(\alpha\alpha_{\epsilon}\right)}\right)\rightarrow\frac{1}{2}$
as $\epsilon\rightarrow0$ gives the result.
\end{proof}
By the last \prettyref{lem:Entropy inequality}, for each $\lambda>\frac{2}{n}\mathtt{h}_{\textrm{top}}$
we have 
\[
L_{d}\leq e^{\lambda},\quad\textrm{ for some }d\in\mathcal{D}_{\frac{1}{2}-}\left(X\right).
\]
By the semi-group property of the flow $\exists c>0$ such that 
\[
d\left(e^{tR}x_{1},e^{tR}x_{2}\right)\leq ce^{\lambda t}d\left(x_{1},x_{2}\right).
\]
From $d^{g^{TX}}\lesssim d\lesssim\left(d^{g^{TX}}\right)^{\frac{1}{2}-\epsilon}$
this further implies that for each $\lambda>\frac{2}{n}\mathtt{h}_{\textrm{top}}$
, $\exists c>0$ such that 
\begin{align}
d^{g^{TX}}\left(e^{tR}x_{1},e^{tR}x_{2}\right) & \leq ce^{2\lambda t}d^{g^{TX}}\left(x_{1},x_{2}\right)\nonumber \\
\left\Vert \left(e^{tR}\right)^{*}f\right\Vert _{C^{2}} & \leq ce^{2\lambda t}\left\Vert f\right\Vert _{C^{2}}\label{eq:growth of distance}
\end{align}
$\forall x_{1},x_{2}\in X$, $f\in C^{2}\left(X\right)$. 

Using the above we shall now prove have the exponential bounds on
the volume of the recurrence set
\begin{align}
\mu\left(S_{T,\varepsilon}\right) & =O\left(\varepsilon^{n}e^{2n\lambda T}\right)\label{eq:recurrence measure Anosov}\\
\mu\left(S_{T,\varepsilon}^{e}\right) & =O\left(\varepsilon^{n}e^{2n\lambda T}\right),\label{eq:recurrence measure extended Anosov}
\end{align}
following an argument in \cite{Dyatlov-Zworski2016}. The above recurrence
set has an obvious lift 
\begin{equation}
\tilde{S}_{T,\varepsilon}\coloneqq\left\{ \left(x,t\right)|t\in\left[\frac{1}{2}T_{0},T\right]\textrm{ s.t. }d^{g^{TX}}\left(e^{tR}x,x\right)\leq\varepsilon\right\} \subset X\times\mathbb{R}\label{eq:lift of the recurrence set}
\end{equation}
satisfying $\pi_{X}\left(\tilde{S}_{T,\varepsilon}\right)=S_{T,\varepsilon}$
under the projection onto the first $X$ factor. The volume bounds
\prettyref{eq:recurrence measure Anosov}, \prettyref{eq:recurrence measure extended Anosov}
shall then follow from \prettyref{eq:growth of distance} and the
following proposition.
\begin{prop}
\label{prop:Anosov lifted recurrence bound} For each $\lambda>\frac{2}{n}\mathtt{h}_{\textrm{top}}$
, the lift $\tilde{S}_{T,\varepsilon}$ \prettyref{eq:lift of the recurrence set}
satisfies the volume estimate
\begin{equation}
\mu_{X\times\mathbb{R}}\left(\tilde{S}_{T,\varepsilon}\right)=O\left(\varepsilon^{n}e^{2n\lambda T}\right)\label{eq:measure estimate on lift}
\end{equation}
with respect to the product measure on $X\times\mathbb{R}$.
\end{prop}

\begin{proof}
First we claim that there exist $C,\delta>0$ of the following significance:
for each pair $\left(x,t\right),\left(x',t'\right)\in\tilde{S}_{T,\varepsilon}$
satisfying $\left|t-t'\right|\leq\delta$, $d^{g^{TX}}\left(x,x'\right)\leq\delta e^{-2\lambda t}$
one has 
\begin{equation}
\left|t-t'\right|\leq C\varepsilon,\quad d^{g^{TX}}\left(x,\cup_{t\in\left[-1,1\right]}e^{tR}x'\right)\leq C\varepsilon.\label{eq:orbits are close}
\end{equation}
By choosing $\delta$ sufficiently small and using \prettyref{eq:growth of distance}
we work in a sufficiently small geodesic coordinate chart in Euclidean
space. The Reeb direction and $E^{u}\oplus E^{s}$ being transverse,
we may replace $x'$ by $e^{tR}\left(x'\right)$, $t\in\left[-1,1\right]$,
to arrange $x-x'\in E^{u}\left(x\right)\oplus E^{s}\left(x\right)$.
Using \prettyref{eq:growth of distance} and a Taylor expansion in
$x,t$ we obtain 
\begin{align*}
\left|e^{tR}\left(x\right)-e^{t'R}\left(x'\right)-de^{tR}\left(x\right)\left(x-x'\right)-R\left(e^{tR}\left(x'\right)\right)\left(t-t'\right)\right| & \leq c_{3}e^{2\lambda t}\left|x-x'\right|^{2}+c_{3}\left|t-t'\right|^{2}\\
 & \leq c_{3}\delta\left|x-x'\right|+c_{3}\delta\left|t-t'\right|.
\end{align*}
Since $\left(x,t\right),\left(x',t'\right)\in\tilde{S}_{T,\varepsilon}$
the above gives 
\begin{align*}
c_{3}\delta\left|x-x'\right|+c_{3}\delta\left|t-t'\right|+2\varepsilon & \geq\left|\left(I-de^{tR}\left(x\right)\right)\left(x-x'\right)-R\left(e^{tR}\left(x'\right)\right)\left(t-t'\right)\right|\\
 & \geq c_{4}\left|x-x'\right|+c_{4}\left|t-t'\right|
\end{align*}
with the second line above following from the Anosov property. It
then remains to choose $\delta$ sufficiently small in relation to
$c_{3},c_{4}$ to obtain \prettyref{eq:orbits are close}.

Finally, let $x_{j}$, $j=1,\ldots N$, be a maximal set of points
such that $d^{g^{TX}}\left(x_{i},x_{j}\right)\geq\delta e^{-2\lambda T}$.
As the balls $\left\{ B_{\frac{\delta e^{-2\lambda T}}{2}}\left(x_{j}\right)\right\} _{j=1}^{N}$
centered at these points are disjoint, the bound $N\leq c_{5}e^{2n\lambda T}$
follows by a computation of the total volume. Furthermore the sets
\begin{align*}
B_{j,k} & \coloneqq B_{2\delta e^{-2\lambda T}}\left(x_{j}\right)\times\left[\frac{1}{2}T_{0}+k\delta,\frac{1}{2}T_{0}+\left(k+1\right)\delta\right],\\
S_{j,k} & \coloneqq\tilde{S}_{T,\varepsilon}\cap B_{j,k},\qquad j=1,\ldots N,\quad k=0,\ldots,1+\left[\delta^{-1}T\right],
\end{align*}
cover $X\times\left[\frac{1}{2}T_{0},T\right]$ and $\tilde{S}_{T,\varepsilon}$
respectively. By \prettyref{eq:orbits are close} small $O\left(\varepsilon\right)$
size neighborhoods of the orbits
\[
\left(\underbrace{\frac{1}{2}T_{0}+\left(k+\frac{1}{2}\right)\delta}_{\eqqcolon t_{k}},\cup_{t\in\left[-1,1\right]}e^{\left(t_{k}+t\right)R}\left(x_{j}\right)\right)
\]
 of volume $O\left(\varepsilon^{n}\right)$, then cover $\tilde{S}_{T,\varepsilon}^{d}$
proving \prettyref{eq:measure estimate on lift}.
\end{proof}

\subsection{\label{subsec:Irrational-elliptic-flows} Elliptic flows}

We now look at elliptic flows on Lens spaces. Given non-negative integers
$q_{0},q_{1},\ldots,q_{m}$, $q_{0}>1$, as well as positive reals
$a_{0},\ldots,a_{m}$ such that $\left(a_{0}^{-1}a_{1},\ldots,a_{0}^{-1}a_{m}\right)\notin\mathbb{Q}^{m}$,
the Lens space is defined as the quotient
\begin{align}
X & =L\left(q_{0},q_{1},\ldots,q_{m};a_{0},\ldots a_{m}\right)\coloneqq E\left(a_{0},\ldots,a_{m}\right)/\mathbb{Z}_{q_{0}},\quad\textrm{ where }\label{eq: irrational lens space-1}\\
E\left(a_{0},\ldots,a_{m}\right) & \coloneqq\left\{ \left(z_{0},\ldots,z_{m}\right)\in\mathbb{C}^{m+1}|\sum_{j=0}^{m}a_{j}\left|z_{j}\right|^{2}=1\right\} \label{eq:irrational ellipsoids-1}
\end{align}
is the irrational ellipsoid. The $\mathbb{Z}_{q_{0}}$ action on the
ellipsoid above is given by $e^{\frac{2\pi i}{q_{0}}}\left(z_{0},\ldots,z_{m}\right)=\left(e^{\frac{2\pi i}{q_{0}}}z_{0},e^{\frac{2\pi iq_{1}}{q_{0}}}z_{1},\ldots,e^{\frac{2\pi iq_{m}}{q_{0}}}z_{m}\right)$. 

The contact form is chosen to be 
\begin{equation}
a=\left.\sum_{j=0}^{m}\left(x_{j}dy_{j}-y_{j}dx_{j}\right)\right|_{E\left(a_{0},\ldots,a_{m}\right)},\quad z_{j}=x_{j}+iy_{j},\label{eq:tautological one form-1}
\end{equation}
the restriction of the tautological form on $\mathbb{C}^{m+1}$, while
its Reeb vector field is computed
\[
R=\sum_{j=0}^{m}a_{j}\left(x_{j}\partial_{y_{j}}-y_{j}\partial_{x_{j}}\right).
\]
Both of the above are seen to be $\mathbb{Z}_{q_{0}}$-invariant and
hence descending to the Lens space quotient. The Reeb flow is then
computed
\begin{align*}
e^{tR}\underbrace{\left[\left(z_{0}\left(0\right),\ldots,z_{m}\left(0\right)\right)\right]}_{\eqqcolon x} & =\left[\left(e^{ia_{0}t}z_{0}\left(0\right),\ldots,e^{ia_{m}t}z_{m}\left(0\right)\right)\right].
\end{align*}
The derivatives of the above flow are seen to be uniformly bounded
$\left|\partial_{x}^{\alpha}e^{tR}\left(x\right)\right|\leq C$ for
all time $t$, the maximal expansion rate $\Lambda_{\textrm{max}}=0$,
and the restriction concerning Ehrenfest time in \prettyref{thm: eta semiclassical limit}
is not necessary. It is then an easy exercise to show that 
\[
e^{tR}x=x\iff\left(\underbrace{a_{0}q_{0}}_{=\tilde{a}_{0}},\underbrace{a_{1}-a_{0}q_{1}}_{=\tilde{a}_{1}},\ldots,\underbrace{a_{m}-a_{0}q_{m}}_{=\tilde{a}_{m}}\right)\in\mathbb{Z}^{m+1}.
\]
Or more generally 
\begin{align*}
d^{g^{TX}}\left(e^{tR}x,x\right) & \thickapprox\sum_{j=1}^{n}\left|e^{i\tilde{a}_{j}t}-1\right|\left|z_{j}\left(0\right)\right|\\
 & \geq Ct^{1-\nu}\left(\min_{j}\left|z_{j}\left(0\right)\right|\right)
\end{align*}
from \prettyref{eq:index tuple} where $\nu=\nu\left(\tilde{a}_{0},\ldots,\tilde{a}_{m}\right)=\nu\left(a_{0},a_{1}-\frac{q_{1}}{q_{0}}a_{0},\ldots,a_{m}-\frac{q_{m}}{q_{0}}a_{0}\right)$
denotes the irrationality index of the given tuple. Hence for $x\in S_{T,\varepsilon}$,
$S_{T,\varepsilon}^{e}$, one has
\[
\min_{j}\left|z_{j}\left(0\right)\right|<C\varepsilon T^{\nu-1}
\]
from which we obtain
\begin{align}
\mu\left(S_{T,\varepsilon}\right) & =O\left(\varepsilon^{2}T^{2\left(\nu-1\right)}\right)\label{eq:recurrence measure elliptic}\\
\mu\left(S_{T,\varepsilon}^{e}\right) & =O\left(\varepsilon^{2}T^{2\left(\nu-1\right)}\right)\label{eq:recurrence measure extended elliptic}
\end{align}
as the recurrence set bounds in this case.

\subsection{\label{subsec:Proofs-of-corollaries}Proofs of corollaries}

We now prove the Corollaries \prettyref{cor: Cor-DG}, \prettyref{cor:Cor-Berard},
\prettyref{cor:Cor -Irr}. These are seen to follow easily from the
main \prettyref{thm: eta semiclassical limit} along with the volume
estimates on the recurrence sets from \prettyref{sec:Examples-of-recurrence}.
\begin{proof}[Proofs of Corollaries \prettyref{cor: Cor-DG}, \prettyref{cor:Cor-Berard}
and \prettyref{cor:Cor -Irr}]
 For Corollary \prettyref{cor: Cor-DG} we choose $\varepsilon=c,T$
to be $h$-independent (and thus in particular $\delta=0$). The formula
\prettyref{eq: eta formula-1} now follows easily from \prettyref{eq: eta formula}
by letting $c\rightarrow0$ and subsequently $T\rightarrow\infty$.

To prove Corollary \prettyref{cor:Cor-Berard}, first choose any $\lambda>\frac{2}{n}\mathtt{h}_{\textrm{top}}$.
And set $\varepsilon=h^{\delta}$, with $\delta=\frac{1}{4}$, and
$T=\frac{1}{8\lambda}\left|\ln h\right|$ in Theorem 1. The equation
\prettyref{eq:growth of distance} shows that $\Lambda_{\textrm{max}}\leq2\lambda$;
and hence for this choice the condition $T\leq T_{E}^{\ell}\left(h\right)$,
$\ell>0$, regarding the Ehrenfest time in Theorem 1 is satisfied.
Furthermore, equation \prettyref{eq:recurrence measure extended Anosov}
says $\mu\left(S_{T,\varepsilon}^{e}\right)=O\left(\varepsilon^{n}e^{2n\lambda'T}\right)$
for any other $\lambda>\lambda'>\frac{2}{n}\mathtt{h}_{\textrm{top}}$.
Thus Theorem 1 gives
\begin{align*}
\left|R\left(h\right)\right| & \leq\frac{\det\left|\mathfrak{J}\right|}{\left(4\pi\right)^{n/2}}T^{-1}+O\left(T^{-2}+\mu_{g^{TX}}\left(S_{T,\varepsilon}^{e}\right)\right)\\
 & =\frac{\det\left|\mathfrak{J}\right|}{\left(4\pi\right)^{n/2}}\left(8\lambda\right)\left|\ln h\right|^{-1}+O\left(\left|\ln h\right|^{-2}+h^{\frac{n}{4}\left(1-\frac{\lambda'}{\lambda}\right)}\right).
\end{align*}
The last line above gives Corollary 3 \prettyref{eq:remainder estimate with entropy}
on letting $\lambda\rightarrow\frac{2}{n}\mathtt{h}_{\textrm{top}}$. 

Finally for Corollary \prettyref{cor:Cor -Irr}, as noted in \prettyref{subsec:Irrational-elliptic-flows},
the derivatives of the Reeb flow on the Lens spaces are uniformly
bounded $\left|\partial_{x}^{\alpha}e^{tR}\left(x\right)\right|\leq C$
for all time $t$. Hence the restriction concerning Ehrenfest time
in \prettyref{thm: eta semiclassical limit}, as required by \prettyref{lem: Lemma req Ehrenfest time},
is not necessary.  We may then set $\varepsilon=h^{\frac{1}{2}-}$,
$T=h^{-\frac{1}{2\nu-1}}$ in \prettyref{thm: eta semiclassical limit}.
The estimate \prettyref{eq:recurrence measure elliptic} on the recurrence
set then estimates the remainder \prettyref{eq:general remainder}
\[
R\left(h\right)=O\left(h^{\frac{1}{2\nu-1}-}\right)
\]
as required.
\end{proof}
\begin{rem}
\label{rem:Entropy conjecture} Conjecturally it is reasonable to
expect 
\begin{equation}
n\left(\inf_{d\in\mathcal{D}_{1}\left(X\right)}\ln L_{d}\right)=\mathtt{h}_{\textrm{top}}\label{eq:stronger entropy conjecture}
\end{equation}
as a characterization of topological entropy, and a strengthening
of the entropy inequality in \prettyref{lem:Entropy inequality}.
If the above \prettyref{eq:stronger entropy conjecture} were true,
it would imply the sharper exponential bounds
\begin{align}
\mu\left(S_{T,\varepsilon}\right) & =O\left(\varepsilon^{n}e^{n\lambda T}\right)\label{eq:recurrence measure Anosov-1}\\
\mu\left(S_{T,\varepsilon}^{e}\right) & =O\left(\varepsilon^{n}e^{n\lambda T}\right),\label{eq:recurrence measure extended Anosov-1}
\end{align}
$\forall\lambda>\frac{1}{n}\mathtt{h}_{\textrm{top}}$, than \prettyref{eq:recurrence measure Anosov}
\prettyref{eq:recurrence measure extended Anosov}, on the recurrence
set volumes of an Anosov flow. This in turn would improve the constant
in the remainder estimate \prettyref{eq:remainder estimate with entropy}
of Corollary \prettyref{cor:Cor-Berard} by a factor of four 
\[
\left|R\left(h\right)\right|\leq\mathtt{h}_{\textrm{top}}\left(\frac{4}{n}\frac{\det\left|\mathfrak{J}\right|}{\left(4\pi\right)^{n/2}}\right)\left|\ln h\right|^{-1}+o\left(\left|\ln h\right|^{-1}\right)
\]
using the same argument as above.
\end{rem}

\section{\label{sec:Quantum-ergodicity}Quantum ergodicity}

In this section we shall prove quantum ergodicity for the coupled
Dirac operator $D_{h}$ \prettyref{eq:Semiclassical Magnetic Dirac}.
The arguments of this section are valid under the weaker assumption
of suitability (rather than strong suitability) for the metric $g^{TX}$
from \cite{Savale2017-Koszul}. This means that the spectrum of the
contracted endomorphism $\textrm{Spec}\left(\mathfrak{J}_{x}\right)=\left\{ 0\right\} \cup\left\{ \pm i\mu_{j}\nu\left(x\right)\right\} _{j=1}^{m}$
is equi-proportinal rather than constant, i.e. allowed to vary with
a single smooth function $\nu\in C^{\infty}\left(X\right)$ on the
manifold. As already noted in the introduction of \cite{Savale2017-Koszul},
this weaker assumption is satisfied by all metrics in dimension 3. 

We now state the result. For $a,b\in\mathbb{R}$ let
\begin{align}
E_{h}\left(a,b\right) & \coloneqq\bigoplus_{\lambda\in\left[a\sqrt{h},b\sqrt{h}\right]}\textrm{ker }\left[D_{h}-\lambda\right]\nonumber \\
N_{h}\left(a,b\right) & \coloneqq\textrm{dim }E_{h}\left(a,b\right)\label{eq:Weyl counting function}
\end{align}
denote the span of the eigenspaces for $D_{h}$ with eigenvalues in,
and the number of eigenvalues in, a given $\sqrt{h}$ size interval
$\left[a\sqrt{h},b\sqrt{h}\right]$. Further for each $h\in\left(0,1\right]$
we choose an orthonormal basis $\left\{ \varphi_{j}^{h}\right\} _{j=1}^{N_{h}}$
for $E_{h}\left(a,b\right)$. A family of subsets $\Lambda_{h}\subset\left\{ 1,2,\ldots,N_{h}\right\} $,
$h\in\left(0,1\right]$, is said to be of density one if it satisfies
$\lim_{h\rightarrow0}\frac{\#\Lambda_{h}}{N_{h}}=1$. 

The leading Weyl asymptotics for $N_{h}\left(a,b\right)$ is obtained
from \prettyref{eq: local on diagonal expansion}. The pointwise trace
$u_{0,x}\left(s\right)\coloneqq\textrm{tr }U_{0,x}\left(s\right)$
of its leading term was computed in \cite[Prop. 7.4]{Savale2017-Koszul}
and is seen to be locally integrable in $s$. A Tauberian argument
following \cite[Prop. 7.1]{Savale2017-Koszul} then gives 
\begin{align}
N_{h}\left(a,b\right) & \sim h^{-n/2}\int_{X}\mu_{a,b}\nonumber \\
\textrm{for }\quad\mu_{a,b} & \coloneqq\left(\int_{a}^{b}ds\,u_{0,x}\left(s\right)\right)dx.\label{eq:Weyl measure}
\end{align}

Our goal in this section will be to prove the following theorem.
\begin{thm}
\label{thm:QE theorem} Let $a$ be a contact form and $g^{TX}$ a
suitable metric. Assume that the Reeb flow of $a$ is ergodic. 

Then one has quantum ergodicity (QE): there exists a density one family
of subsets $\Lambda_{h}\subset\left\{ 1,2,\ldots,N_{h}\right\} $,
$h\in\left(0,1\right]$, such that 
\[
\left\langle B\varphi_{j}^{h},\varphi_{j}^{h}\right\rangle \rightarrow\int_{X}\left(\left.b_{0}\right|_{\Sigma}\right)\mu_{a,b},
\]
$j\in\Lambda_{h}$, as $h\rightarrow0$ for each $B\in\Psi_{{\rm cl\,}}^{0}\left(X\right)$,
with homogeneous principal symbol $b_{0}=\sigma\left(B\right)\in C^{\infty}\left(T^{*}X\right)$.
In particular, the eigenfunctions get equidistributed $\left|\varphi_{j}^{h}\right|^{2}dx\rightharpoonup\mu_{a,b}$,
$j\in\Lambda_{h}$, as $h\rightarrow0$ according to the measure \prettyref{eq:Weyl measure}.
\end{thm}

\begin{proof}
Following a general outline for quantum ergodicity theorems (see for
example \cite[Sec. 15.4]{Zworski}) it suffices to prove a microlocal
Weyl law 
\begin{align}
E\left(B\right)\coloneqq & \lim_{h\rightarrow0}\frac{1}{N_{h}\left(a,b\right)}\sum_{j=1}^{N_{h}}\left\langle B\varphi_{j}^{h},\varphi_{j}^{h}\right\rangle \nonumber \\
= & \int_{X}\left(\left.b_{0}\right|_{\Sigma}\right)\mu_{a,b},\label{eq: micro-weyl}
\end{align}
and a variance estimate 
\begin{equation}
V\left(B\right)\coloneqq\lim_{h\rightarrow0}\frac{1}{N_{h}\left(a,b\right)}\sum_{j=1}^{N_{h}}\left|\left\langle \left[B-E\left(B\right)\right]\varphi_{j}^{h},\varphi_{j}^{h}\right\rangle \right|^{2}=0,\label{eq:variance}
\end{equation}
$\forall B\in\Psi_{{\rm cl\,}}^{0}\left(X\right),$ with $b=\sigma\left(B\right)$. 

The microlocal Weyl law follows immediately by integrating the microlocal
trace expansion of \prettyref{thm: microlocal trace expansion}, and
the formula for its leading term \prettyref{eq:leading coefficient},
via a Tauberian argument. 

For the variance estimate, firstly replacing $B$ with $B-E\left(B\right)$,
which has the same variance, one may assume $E\left(B\right)=0$.
An application of Cauchy-Schwartz (cf. \cite[Lemma 4.1]{Colin-de-Verdiere-Hillairet-TrelatI})
gives 
\begin{align}
V\left(B\right) & \leq E\left(B^{*}B\right)\nonumber \\
V\left(B_{2}\right)=0 & \implies V\left(B_{1}\right)=V\left(B_{1}+B_{2}\right).\label{eq:Cauchy-Schwartz =000026 young implications}
\end{align}
The above along with \prettyref{eq: micro-weyl} gives the variance
estimate
\begin{equation}
\left.b_{0}\right|_{\Sigma}=0\implies V\left(B\right)=0\label{eq:vanishing symbol has zero variance}
\end{equation}
for pseudodifferential operators with principal symbol vanishing on
the characteristic variety. 

Next, consider the lift of the Reeb vector field $\hat{R}\in C^{\infty}\left(T\Sigma\right)$
to the characteristic variety, satisfying $\pi_{*}\hat{R}=R$ under
the natural projection $\pi:T^{*}X\rightarrow X$. We now prove the
variance estimate for those pseudodifferential operators whose principal
symbol satisfies 
\begin{equation}
\left.b_{0}\right|_{\Sigma}=\hat{R}\left(\left.a\right|_{\Sigma}\right)c\left(i\hat{R}\right)\label{eq:variance for Reeb differential}
\end{equation}
for some $a\in S^{0}\left(T^{*}X\right)$. To this end, we may use
a partition of unity to suppose that the symbol $a$ is supported
in a small microlocal chart near the characteristic variety. Specifically
we use the microlocal chart in which the normal form for the Dirac
operator \cite[Prop. 5.2]{Savale2017-Koszul} holds near $\Sigma$.
In this chart one has $\Sigma=\left\{ \left(x,\xi\right)|\xi_{0}=\xi_{1}=\ldots\xi_{m}=x_{1}=\ldots x_{m}=0\right\} \subset T^{*}\mathbb{R}_{x}^{n}$
in some phase space variables. One may Taylor expand $a=\underbrace{a_{0}\left(x_{0},x_{m+1}\ldots x_{2m},\xi_{m+1},\ldots\xi_{2m}\right)}_{=\left.a\right|_{\Sigma}}+\underbrace{a_{1}}_{=O_{\Sigma}\left(1\right)}$
modulo a term $O_{\Sigma}\left(1\right)$ vanishing along $\Sigma$.
Following \cite[Prop. 5.2]{Savale2017-Koszul} or \cite[eq. 5.8]{Savale2017-Koszul}
we then compute the commutator
\begin{align}
\left[a_{0}^{W},D_{h}\right] & =hc^{W}\nonumber \\
\textrm{for }\quad c & =\underbrace{\left(\partial_{x_{0}}a_{0}\right)\sigma_{0}}_{=\hat{R}\left(\left.a\right|_{\Sigma}\right)c\left(i\hat{R}\right)}+O_{\Sigma}\left(1\right)+O\left(h\right).\label{eq:computation of commutator}
\end{align}
Identifying the leading term above in terms of the lift of the Reeb
vector field, and the variance of the above commutator being zero,
we use \prettyref{eq:vanishing symbol has zero variance} to obtain
\prettyref{eq:variance for Reeb differential}.

Finally, we choose a smooth family of symbols $b_{t}\in S^{0}\left(T^{*}X;\textrm{End}\left(S\right)\right)$
such that $\left.b_{t}\right|_{\Sigma}=\begin{bmatrix}b_{t}^{1}\\
 & b_{-t}^{1}
\end{bmatrix},$$b_{t}^{1}\coloneqq\left(e^{t\hat{R}}\right)^{*}\left(\left.b_{0}\right|_{\Sigma}\right)$,
in terms of the decomposition $S=S_{+}\oplus S_{-}$ given by the
$\pm i$ eigenspaces of $c\left(\hat{R}\right)$. Setting $B_{t}=b_{t}^{W}$
it then follows from \prettyref{eq:variance for Reeb differential}
that $V\left(\frac{d}{dt}B_{t}\right)=0$. Hence 
\begin{align}
V\left(B_{0}\right)=V\left(B_{T}\right) & =V\left(\underbrace{\frac{1}{2T}\int_{-T}^{T}B_{t}dt}_{\eqqcolon\bar{B}_{T}}\right)\nonumber \\
 & \leq E\left(B_{T}^{*}B_{T}\right)=\int_{X}\left|\left.\bar{b}_{T}\right|_{\Sigma}\right|^{2}\mu_{a,b}\label{eq:final variance computation}
\end{align}
for $\bar{b}_{T}\coloneqq\frac{1}{2T}\int_{-T}^{T}b_{t}dt$, $\forall T$
, from \prettyref{eq:Cauchy-Schwartz =000026 young implications}
and \prettyref{eq: micro-weyl}. Finally the ergodicity of the Reeb
flow is equivalent to that of its lift $\hat{R}$. Hence by an application
of the von Neumann mean ergodic theorem the last line above converges
to zero as $T\rightarrow\infty$ completing the proof.
\end{proof}
\textbf{Acknowledgments.} This article arose out of a conversation
with S. Nonnenmacher who the author thanks for the discussion. The
author would also like to thank A. Abbondandolo, F. Ledrappier, I.
Polterovich, S. Roth and K. War for answering questions and helpful
explanations related to the article. Finally, thanks are due to the
diligent and expert referee for important corrections to an earlier
submission.

\bibliographystyle{siam}
\bibliography{biblio}

\end{document}